\documentclass[11pt,leqno]{siamltex}
\usepackage[top=1in, bottom=1in, left=1in, right=1in]{geometry}
\usepackage{amsfonts}
\usepackage{amsmath}
\usepackage{amssymb}
\usepackage{bm}
\usepackage{subfigure}
\usepackage{graphicx}
\usepackage{algorithm}
\usepackage{algorithmic}
\usepackage{multirow}
\usepackage{xcolor}
\usepackage{makecell}

\usepackage{url}
\newcommand{\Bx}{\bm{x}}

\newcommand{\Bb}{\bm{b}}
\newcommand{\BW}{\bm{W}}

\newcommand{\jg}[1]{{#1}}

\title{Friedrichs Learning: Weak Solutions of Partial Differential Equations via Deep Learning}

\author{Fan Chen
\vspace{0.1in}\\
School of Mathematical Sciences, and MOE-LSC,\\
 Shanghai Jiao Tong University, Shanghai 200240, China  ({\tt alexnwish@sjtu.edu.cn})
\vspace{0.1in}\\
Jianguo Huang
\vspace{0.1in}\\
School of Mathematical Sciences, and MOE-LSC,\\
 Shanghai Jiao Tong University, Shanghai 200240, China  ({\tt jghuang@sjtu.edu.cn})
\vspace{0.1in}\\
Chunmei Wang
\vspace{0.1in}\\
Department of Mathematics,  University of Florida, Gainesville, FL 32611, USA  ({\tt chunmei.wang@ufl.edu})
\vspace{0.1in}\\
Haizhao Yang
  \vspace{0.1in}\\
  Department of Mathematics, University of Maryland, College Park, MD 20742,  USA  ({\tt hzyang@umd.edu})
}

\begin{document}
\maketitle
\begin{abstract}
This paper proposes Friedrichs learning as a novel deep learning methodology that can learn the weak solutions of PDEs via a minimax formulation, which transforms the PDE problem into a minimax optimization problem to identify weak solutions. The name ``Friedrichs learning" is to highlight the close relation between our learning strategy and Friedrichs theory on symmetric systems of PDEs. The weak solution and the test function in the weak formulation are parameterized as deep neural networks in a mesh-free manner, which are alternately updated to approach the optimal solution networks approximating the weak solution and the optimal test function, respectively. Extensive numerical results indicate that our mesh-free Friedrichs learning method can provide reasonably good solutions for a wide range of PDEs defined on regular and irregular domains, where conventional numerical methods such as finite difference methods and finite element methods may be tedious or difficult to be applied, especially for those with discontinuous solutions in high-dimensional problems.
\end{abstract}



\pagestyle{myheadings}
\thispagestyle{plain}
\markboth{Friedrichs Learning for Weak Solutions of PDEs}{Friedrichs Learning for Weak Solutions of PDEs}

\section{Introduction}
High-dimensional PDEs and PDEs defined on complex domains are important tools in physical, financial, and biological models, etc. \cite{Lee2002,Ehrhardt2008,Yserentant2005,Gaikwad2009,Wales2003}. Generally speaking, they do not have closed-form solutions making numerical solutions of such equations indispensable in real applications. First, developing numerical methods for high-dimensional PDEs has been a challenging task due to the curse of dimensionality in conventional discretization. 
Second, conventional numerical methods rely on mesh generation that requires profound expertise and programming skills without the use of commercial software. In particular, for problems defined in complicated domains, it is challenging and time-consuming to implement conventional methods. As an efficient parametrization tool for high-dimensional functions \cite{Barron1993,E2019,Montanelli2019,Montanelli2019_2,SIEGEL2020313,Hutzenthaler2020,Hutzenthaler2020_2,Shen2021,Shen2021_2} with user-friendly software (e.g., TensorFlow and PyTorch), neural networks have been applied to solve PDEs via various approaches recently. The idea of using neural networks to solve PDEs dates back to the 1990s \cite{Lee1990,Gobovic1994,Dissanayake1994,Lagaris1998} and was revisited and popularized recently \cite{E2017,Han2018,E2018,Khoo2017SolvingPP,Sirignano2018,Berg2018,Li2019,Beck2021, Hutzenthaler2020_1,Hutzenthaler2020,Cai2020,RAISSI2019686,Li2019_2,Zang2020,Bao2020,liu2020multiscale,Kharazmi2021,Jagtap2020}.

Many network-based PDE solvers are concerned with the classical solutions that are differentiable and satisfy PDEs in common sense. Unlike classical solutions, weak solutions are functions for which the derivatives may not always exist but which are nonetheless deemed to satisfy the PDE in some precisely defined sense. These solutions are crucial because many PDEs in modeling real-world phenomena do not have sufficiently smooth solutions. 
Motivated by the seminal work in \cite{Bao2020}, we propose Friedrichs learning as an alternative method that can learn the weak solutions of elliptic, parabolic, and hyperbolic PDEs in $L^2(\Omega)$ via a novel minimax formulation devised and analyzed in Section 2.3. Since the formulation is closely related to the work of Friedrichs theory on symmetric systems of PDEs (cf. \cite{Friedrichs1958}), we call our learning strategy the Friedrichs learning. The main idea is to transform the PDE problem into a minimax optimization problem to identify weak solutions. Note that no regularity for the solution is required in Friedrichs learning, which is the main advantage of the proposed method, making it applicable to a wide range of PDE problems, especially those with discontinuous solutions. In addition, Friedrichs learning is capable of solving PDEs with discontinuous solutions without a priori knowledge of the location of the discontinuity. Although Friedrichs learning may not be able to provide highly accurate solutions, it could solve a coarse solution without a priori knowledge of the discontinuity. This rough estimation of the discontinuity could serve as a good initial guess of conventional computation approaches for highly accurate solutions following the Int-Deep framework in \cite{Huang2020}. Finally, theoretical results are provided to justify the Friedrichs learning framework for various PDEs.

The main philosophy of Friedrichs learning is to reformulate a PDE problem into a minimax optimization, the solution of which is a test deep neural network (DNN) that maximizes the loss and a solution DNN that minimizes the loss. For a high-order PDE, we first reformulate it into a first-order PDE system by introducing auxiliary variables, the weak form of which naturally leads to a minimax optimization using integration by parts according to the theory of Friedrichs' system \cite{Friedrichs1958}. The above-mentioned feature is the crucial difference from existing deep learning methods for weak solutions \cite{E2018,Zang2020}.  Let us introduce the formulation of Friedrichs learning using first-order boundary value problems (BVPs) with homogeneous boundary conditions without loss of generality. The initial value problems (IVPs) can be treated as BVPs, where the time variable is considered to be one more spatial variable. The non-homogeneous boundary conditions can be easily transferred to homogeneous ones by subtracting the boundary functions from the solutions. 

In the seminal results by Friedrichs in \cite{Friedrichs1958} and other investigations in \cite{Antonic2010,Ern2007}, an abstract framework of the boundary value problem of the first-order system was established, which is referred to as Friedrichs' system in the literature. Let us introduce the concept of Friedrichs' system using a concrete and simple example and illustrate the main idea and intuition of the Friedrichs learning proposed in this paper. A more detailed abstract framework of Friedrichs learning will be discussed later in Section \ref{sec:FL}. Let $r\in \mathbb N$ and $\Omega\subset \mathbb{R}^d$ be an open and bounded domain with Lipschitz boundary $\partial\Omega$.  The notation $(\cdot)^{\intercal}$  denotes the  transpose of a vector or a matrix throughout the paper.  We assume: 1) $\bm{A}_k \in [L^{\infty}(\Omega)]^{
r\times r}$, $\sum_{k=1}^d \partial_k \bm{A}_k\in [L^{\infty}(\Omega)]^{r\times r}$, $\bm{A}_k=\bm{A}_k^{\intercal}$ a.e. in $\Omega$ for  $k=1,\dots,d$,  and $\bm{C}\in [L^\infty(\Omega)]^{r\times r}$; 2) the full coercivity holds true, i.e., $\bm{C}+\bm{C}^{\intercal}-\sum_{k=1}^d\partial_k \bm{A}_k\geq 2\mu_0 \bm{I}_r$ a.e. in $\Omega$ for some $\mu_0>0$ and the identity matrix $\bm{I}_r\in \mathbb R^{r\times r}$. Then the first-order differential operator $T:\mathcal{D}  \rightarrow L$  with $L=[L^2(\Omega)]^{r}$ and $\mathcal{D}=[C_0^{\infty}(\Omega)]^r$ defined by $T\bm{u}:=\sum_{k=1}^d \bm{A}_k \partial_k\bm{u}+\bm{C}\bm{u}$ is called the Friedrichs operator, while the first-order system of PDEs $T\bm{u}=\bm{f}$ is called the Friedrichs' system, where $\bm{f}$ is a given data function in $L$ and the space $C_0^{\infty}(\Omega)$ consists of all infinitely differentiable functions with compact support in $\Omega$. Throughout this paper, the bold font will be used for vectors and matrices in concrete examples. In our abstract framework, PDE solutions are considered as elements of a Hilbert space, so they will not be denoted as bold letters.

Friedrichs \cite{Friedrichs1958} also introduced an abstract framework for representing boundary conditions via matrix-valued boundary fields. First, let $\bm{A}_{\bm{n}}:=\sum_{k=1}^d n_k\bm{A}_k\in [L^\infty(\partial \Omega)]^{r\times r}$, where $\bm{n}=(n_1,\cdots,n_d)\in \mathbb{R}^d$ is the unit outward normal direction on $\partial\Omega$, and let $\bm{M}:\partial\Omega\rightarrow \mathbb{R}^{r\times r}$ be a matrix field on the boundary. Then a homogeneous Dirichlet boundary condition of Friedrichs' system is  prescribed by $(\bm{A}_{\bm{n}}-\bm{M})\bm{u}=\bm{0}$ on $\partial\Omega$ by choosing an appropriate $\bm{M}$ to ensure the well-posedness of Friedrichs' system. In real applications, $\bm{M}$ is given by physical knowledge. Let $V:=\mathcal{N}(\bm{A}_{\bm{n}}-\bm{M})$ and $V^*:=\mathcal{N}(\bm{A}_{\bm{n}}+\jg{\bm{M}^{\intercal}})$, where  $\mathcal{N}$ is the null space of the argument.
It has been proved that $\bm{u}$ solves the BVP
\begin{equation}\label{eqn:bvp}
T\bm{u}=\bm{f} \text{ in }\Omega\quad\text{and}\quad (\bm{A}_{\bm{n}}-\bm{M})\bm{u}=\bm{0}\text{ on }\partial\Omega,
\end{equation}
if and only if $\bm{u}$  solves the minimax problem
\[
\min_{\bm u \in V} \max_{\bm v \in  V^*} \mathcal{L}(\bm u,\bm v) := \frac{|(\bm u,\tilde T \bm v)_{L} - (\bm f,\bm v)_{L}|}{\| \tilde T \bm v\|_{L}},
\]
where $\tilde{T}: \mathcal{D}\to L$ is the formal adjoint of $T$. Hence, in our Friedrichs learning, DNNs are applied to parametrize $\bm{u}$ and $\bm{v}$ to solve the above minimax problem to obtain the solution of the BVP \eqref{eqn:bvp}. Friedrichs learning also works for other kinds of boundary conditions.

This paper is organized as follows. In Section \ref{sec:FL}, we devise and analyze Friedrichs minimax formulation for weak solutions of PDEs. In Section \ref{sec:EX}, several concrete examples of PDEs and their minimax formulations are provided. In Section \ref{sec:DL}, network-based optimization is introduced to solve the minimax problem in Friedrichs formulation. In Section \ref{sec:RS}, a series of numerical examples are provided to demonstrate the effectiveness of the proposed Friedrichs learning. Finally, we conclude this paper in Section \ref{sec:CC}.

\section{Friedrichs Minimax Formulation for Weak Solutions}\label{sec:FL}
In this section, we shall first recall some standard notations frequently used later on. Then, briefly review Friedrichs' system in a Hilbert space setting  \cite{Ern2007,Bui-Thanh2013}, followed by introducing and analyzing Friedrichs minimax formulation for weak solutions which is the foundation of Friedrichs learning.

{Let $\Omega\subset \mathbb R^d$ be a bounded domain with the Lipschitz boundary. Let $D_j=\frac{\partial}{\partial x_j}$ be the partial derivative operation with respect to $x_j$ in the weak sense. For a multi-index $\alpha=(\alpha_1,\cdots,\alpha_d)$ with each $\alpha_i$ being a non-negative integer, denote $D^\alpha=D_1^{\alpha_1}D_2^{\alpha_2}\cdots D_d^{\alpha_d}$. For a non-negative integer $k$ and a real number with $1\le p\le \infty$, define the Sobolev space $W^{k,p}(\Omega)$ as a vector space consisting of all functions $v\in L^p(\Omega)$ such that $D^{\alpha} v\in L^p(\Omega)$ for all multi-indices $\alpha$ with $|\alpha|= \sum_{j=1}^d \alpha_j\le k$, which is equipped with the following norm:
$$
\|v\|_{W^{k,p}(\Omega)}=\Big(\sum_{|\alpha|\leq k}\int_{\Omega} |D^{\alpha}u|^pdx\Big)^{1/p}, \quad 1\leq p <\infty;\qquad 
\|u\|_{W^{k, \infty}(\Omega)}=\sum_{|\alpha| \leq k} \operatorname{esssup}_{\Omega}\left|D^\alpha u\right|,
$$
where $\operatorname{esssup}_{\Omega}$ is the essential supremum for a function in $\Omega$. When $p=2$, $W^{k,2}(\Omega)$ is simply written as $H^k(\Omega)$. In addition, let $H_0^k(\Omega)$ be the closure of $C_0^{\infty}(\Omega)$ with respect to the norm of $H^k(\Omega)$, while $H^{-k}(\Omega)$ denotes the dual space of $H^k_0(\Omega)$.  We refer the reader to the monograph \cite{Adams1975} for details about Sobolev spaces and their properties.}

{Let $L$ denote a real Hilbert space, which is equipped with the inner product $(\cdot,\cdot)_L$ and the induced norm $\|\cdot\|_{L}$. For any two vectors in an Euclidean space, we use $(\cdot,\cdot)$ to represent their natural inner product and denote by  $\left\|\cdot\right\|_p$ the related $\ell_p$ norm for $1\le p \le \infty$; most of these symbols will appear in Sections \ref{sec:DL} and \ref{sec:RS}. For a vector space $W$ and its dual space $W'$, the notation $\langle\cdot,\cdot\rangle_{W\times W'}$ represents the duality pair between $W$ and $W'$. For any two Hilbert spaces $X$ and $Y$, denote by ${\mathcal L}(X,Y)$ the vector space consisting of all continuous linear operators from $X$ into $Y$. }


\subsection{An Abstract Framework of Friedrichs' System}
First of all, we recall some basic results on Friedrichs' system developed in  \cite{Bui-Thanh2013, Ern2007} for later use in order to be self-contained. Let $L$ be a real Hilbert space, and dual space of $L$, denoted by $L'$, can be identified naturally with $L$ by the Riesz representation theorem. For a dense subspace ${\mathcal D}$ of $L$, we consider two linear operators $T:{\mathcal D}\to L$ and $\tilde{T}: {\mathcal D}\to L$ satisfying the following properties: for any ${u}, {v}\in {\mathcal D}$, there exists a positive constant $C$ such that
\begin{eqnarray}
    (T{u}, {v})_L&=&({u},\tilde{T}{v})_L,\label{p1}\\
    \|(T+\tilde{T}){u}\|_L &\leq& C\|{u}\|_L.\label{p2}
\end{eqnarray}
{It is worth noting that the two operators $T$ and $\tilde{T}$ are given simultaneously. Due to the property \eqref{p1}, we often call $\tilde{T}$ as the formal adjoint of $T$ and vice versa. Since the operators $T$ and $\tilde{T}$ play the same roles, we will focus on the forthcoming discussion for $T$, which can be applied to $\tilde{T}$ in a straightforward way.}  As shown in \cite[Sect. 5.5]{Aubin2000}, write $W_0$ as the completion of ${\mathcal D}$ with respect to the scalar product {$(\cdot,\cdot)_T=(\cdot,\cdot)_L+(T\cdot,T\cdot)_L$}. Then, we have by \eqref{p1} that
\[
{\mathcal D}\subset W_0\subset L=L^{\prime}\subset W_0^{\prime}\subset{ {\mathcal D^{\prime}}}.
\]
In addition, in view of \eqref{p2}, we know $W_0$ is also the completion of ${\mathcal D}$ with respect to the scalar product {$(\cdot,\cdot)_{\tilde{T}}=(\cdot,\cdot)_{\tilde{L}}+(T\cdot,T\cdot)_{\tilde{L}}$}. Thus, $\tilde{T}$ can be extended from ${\mathcal D}$ to $W_0$, and its true adjoint $(\tilde{T})^{*}\in \mathcal{L}(L;\ W_0^{\prime})$ can be viewed as the extension of $T$ to $L$. When there is no confusion caused, we still use the notation $T$ for this extension operator. This argument applies to $\tilde{T}$ as well.

We provide an example to make the above abstract treatment more accessible. Let $\Omega=(a,b)$. Choose ${\mathcal D}=C_0^{\infty}(\Omega)$ and $L=L^2(\Omega)$. Let $Tv=v^{\prime}$ and $\tilde{T}v=-v^{\prime}$ for all $v\in C_0^{\infty}(\Omega)$. {In this case, we have
	\[
	(v,w)_T=(v,w)_{\tilde{T}}=\int_a^b (vw+v^{\prime}w^{\prime})dx,\quad \forall\; v, w\in C_0^{\infty}(\Omega),
	\]
	so, by definition, the completion of $C_0^{\infty}(\Omega)$ with respect to the induced norm is exactly the Sobolev space $H_0^1(\Omega)$.
{Hence, according to Theorem 1.4.4.6 in \cite[p. 31]{Grisvard1985}, if we understand the derivative operator in the sense of distributions, we know  $(\tilde{T})^{*}\in \mathcal{L}(L^2(\Omega);\ H^{-1}(\Omega))$.} In other words, the derivative operator $(\cdot)^{\prime}$ can be viewed as a continuous linear operator from $L^2(\Omega)$ into $H^{-1}(\Omega)$.}


Next, as given in \cite[Lemma 2.1]{Ern2007}, define a graph space $W$ by
\begin{equation}
\label{eqn:graph}
W=\{u\in L;\; Tu \in L \},
\end{equation}
which is a Hilbert space with respect to the graph norm $\|\cdot\|_T=(\cdot, \cdot)_T ^{1/2}$. In addition, owing to \eqref{p2}, we have
\[
W= \{u \in L;\; \tilde Tu \in L\}.
\]
That means $W$ is also a graph space associated with $\tilde{T}$.

The abstract framework of Friedrichs' system concerns the solvability of the problem
\begin{equation}\label{tu}
    T{u}={f}\in L,
\end{equation}
and its solution falls in the graph space $W$. Obviously, the problem \eqref{tu} may not be well-posed since its solution in $W$ may not be unique. We are interested in constructing a subspace $V\subseteq W$ such that $T : V \to L$ is an isomorphism. A standard way is carried out as follows. We first define a self-adjoint boundary operator $B\in \mathcal{L}(W, W')$ as follows (cf. \cite{Ern2007}):
\begin{equation}\label{bterm}
    \langle B{u}, {v}\rangle_{W'\times W}=(T{u}, {v})_L-({u}, \tilde{T}{v})_L,\quad \forall\; u, v\in W.
\end{equation}

This operator plays a key role in the forthcoming analysis. Moreover, the identity \eqref{bterm} can be reformulated in the form
\[
(T{u}, {v})_L=({u}, \tilde{T}{v})_L+\langle B{u}, {v}\rangle_{W'\times W},
\]
which is usually regarded as an abstract integration by parts formula (cf. \cite{Ern2007}).

Furthermore, we assume that there exists an operator $M\in \mathcal{L}(W, W')$ such that
\begin{eqnarray}\label{pp1}
\langle Mw,w\rangle_{W'\times W}\geq 0,\quad
\forall\, w\in W,\\
W=\mathcal{N}(B-M)+\mathcal{N}(B+M),\label{pp2}
\end{eqnarray}
where $\mathcal{N}$ is the null space of its argument. Meanwhile, let $M^*\in \mathcal{L}(W, W')$ denote the adjoint operator of $M$ given by
$
\langle M^*u,v \rangle_{W'\times W}=\langle Mv,u \rangle_{W'\times W},   \forall\, u,v\in W.
$

To find $V$ such that the problem \eqref{tu} is well-posed, we should make an additional assumption for $L$ as follows; i.e.,
\begin{equation}
\label{p3}
((T+\tilde{T})v,v)_L\ge 2\mu_0 \|v\|^2_L,\quad \forall\, v\in L,
\end{equation}
where $\mu_0$ is a positive constant. Then we choose
\begin{equation}
\label{selection}
V=\mathcal{N}(B-M),\quad V^*=\mathcal{N}(B+M^*).
\end{equation}

We have the following important theory for Friedrichs' system \cite[Lemma 3.2 and Theorem 3.1]{Ern2007}.

\begin{theorem}
\label{key-FS}
Assume \eqref{p2} \eqref{p3},   \eqref{pp1} and \eqref{pp2} hold true. Let $V$ and $V^*$ be given by  \eqref{selection}. The following statements hold true:
\begin{enumerate}
\item For any $v\in W$, it holds
\begin{equation}
\label{a-coercive}
\mu_0\|v\|_L\le \|T v\|_L,\quad \mu_0\|v\|_L\le \|\tilde{T}v\|_L.
\end{equation}

\item For any $f\in L$, problem \eqref{tu} has a unique solution in $V$. In other words, $T$ is an isomorphism from $V$ onto $L$. Moreover, $\tilde{T}$ is an isomorphism from $V^*$ onto $L$.
\end{enumerate}
\end{theorem}

\subsection{First Order PDEs of Friedrichs Type} \label{PDE-system}
As a typical application of the above framework, we restrict  $L$ to be the space of square integral (vector-valued) functions over an open and bounded domain $\Omega\subset\mathbb{R}^d$ with Lipschitz boundary, $\mathcal{D}$ to be the space of test functions, and $T$ to be a first-order differential operator with its formal adjoint $\tilde{T}$. In particular, we take  $L=[L^2(\Omega)]^r$, $r\in \mathbb{N}$ and $\mathcal{D}=[C_0^{\infty} (\Omega)]^r$.   $\mathcal{D}$ is thus dense in $L$.
Consider $T: \mathcal{D} \to L$ as follows
\begin{equation}\label{eqn}
T\bm{u}=\sum_{k=1}^d \bm{A}_k\partial_k \bm{u}+\bm{C}\bm{u}=\bm{f}, \quad \forall\, \bm{u}\in \mathcal{D}.
\end{equation}
The  standard assumptions are imposed on $\bm{A}_k$ and $\bm{C}$ for Friedrichs' system   \cite{ern200601,ern200602,Friedrichs1958}:

\begin{eqnarray}\label{pa}
\bm{C}&\in & [L^{\infty}(\Omega)]^{r\times r},\\
\bm{A}_k & \in &[L^{\infty}(\Omega)]^{r\times r},k=1,\cdots,d \quad \text{and}\quad \sum_{k=1}^d\partial_k\bm{A}_k\in [L^{\infty}(\Omega)]^{r\times r}\label{pb}\\
\bm{A}_k&=&\bm{A}_k^{\intercal},\quad  \text{a. e. in} \  \Omega,\; k=1,\cdots, d. \label{pc}
\end{eqnarray}
The formal adjoint $\tilde{T}: \mathcal{D}\to L$ of $T$ can be defined by
\begin{equation}
\label{adjoint}
    \tilde{T}\bm{u}=-\sum_{k=1}^d \bm{A}_k\partial_k\bm{u}+(\bm{C}^{\intercal}-\sum_{k=1}^d\partial_k\bm{A}_k)\bm{u}, \quad \forall\, \bm{u}\in \mathcal{D}.
\end{equation}

It is easy to see that $T$ and $\tilde{T}$ satisfy \eqref{p1}-\eqref{p2}. All the
results in this section hold true for Friedrichs' system satisfying \eqref{pa}-\eqref{pc}.

For an abstract Friedrichs' system, one may find the explicit representation of $B$, but it is very difficult to derive the operator on $M$ which is governed by the conditions \eqref{pp1} and \eqref{pp2}. Assume $\mathcal{B}=\sum_{k=1}^d n_k\bm{A}_k$
is well-defined a.e. on $\partial \Omega$ where $\bm{n}=(n_1,\cdots,n_d)^{\intercal}$ is the unit outward normal vector of $\partial \Omega$. For simplicity of notations, we set $\mathcal{H}^s=[H^s]^r$ with $H^s$ being the usual Sobolev space of order $s$, and $\mathcal{C}^1=[C^1]^r$ with $C^1$ being the space of continuously differentiable
functions, {similarly notate $\mathcal C_0^{\infty} = [C_0^{\infty}]^r$.}

\begin{lemma}\cite{Anthony2009book,Jensen2004}
For $\bm{u}, \bm{v}\in \mathcal{H}^1(\Omega)\subset W(\Omega)$, there holds
$$
\langle B\bm{u}, \bm{v}\rangle_{W'(\Omega)\times W(\Omega)}=\langle \mathcal{B}\bm{u},\bm{v}\rangle_{\mathcal{H}^{-\frac{1}{2}}(\partial\Omega)\times \mathcal{H}^{\frac{1}{2}}(\partial\Omega)},
$$
where $ W(\Omega)= \{\bm u \in L(\Omega);\; T \bm u \in L(\Omega)\}$ and $W'(\Omega)$ is the dual space of $W(\Omega)$. Specifically, $\langle B\bm{u}, \bm{v}\rangle_{W'(\Omega)\times W(\Omega)}=\int_{\partial \Omega} \bm{v}^{\intercal}\mathcal{B}\bm{u} ds$, for any $\bm{u}, \bm{v}\in \mathcal C_0^{\infty} (\mathbb{R}^d)$.
\end{lemma}

If $\Omega$ has segment property \cite{Antonic2009}, $\mathcal{C}^1{ (\overline{\Omega}) }$ is thus dense in $\mathcal{H}^1(\Omega)$  and further is dense in $W(\Omega)$. Therefore, the representation could be uniquely extended to the whole space $W(\Omega)$ in the sense that for any $\bm u\in W(\Omega)$ and $\bm v\in \mathcal{H}^1(\Omega)$,
\begin{equation}
    \langle B\bm u, \bm v\rangle_{W'(\Omega)\times W(\Omega)}=\langle \mathcal{B}\bm u, \bm v\rangle_{\mathcal{H}^{-\frac{1}{2}}(\partial\Omega)\times \jg{\mathcal{H}^{\frac{1}{2}}}(\partial\Omega)}.
\end{equation}

The coercivity condition on $T$ dictated by the positiveness condition on the coefficients $\bm{A}_k$
and $\bm{C}$ \cite{ern200601,ern200602,ern2008} is needed to show the well-posedness of PDEs of Friedrichs type. After some direct manipulation, the abstract coercivity condition \eqref{p3} is equivalent to the following full coercivity for Friedrichs  PDEs:
\begin{equation}\label{coercive}
    \bm{C}+\bm{C}^{\intercal}-\sum_{k=1}^d \partial_k\bm{A}_k\geq 2\mu_0 \bm{I}_r, \quad \text{a.e., in}\ \Omega,
\end{equation}
where $\mu_0$ is a positive constant and $\bm I_r$ is the $r\times r$ identity matrix. If a system does not satisfies the coercivity condition \eqref{coercive} we can introduce a feasible transformation so that the modified system satisfies this condition. In \cite{Bui-Thanh2013}, the authors introduced the so-called partial coercivity condition to study the mathematical theory of the corresponding system.
Readers are referred to \cite{Bui-Thanh2013} for more details.

\subsection{Friedrichs Minimax Formulation}
Throughout this subsection, we assume all the conditions given in Theorem \ref{key-FS} hold true. Recall that $V=\mathcal{N}(B-M)$ and $V^*=\mathcal{N}(B+M^*)$ with $M \in \mathcal{L}(W,W')$ satisfying conditions \eqref{pp1}-\eqref{pp2}. For a given $f\in L$, find the solution $u\in V$ such that
\begin{equation}
\label{PDEproblem}
T u = f,
\end{equation}
or equivalently,
\begin{equation}\label{weak}
(Tu,v)_{L} = (f, v)_{L},\quad \forall\, v \in L.
\end{equation}
In most cases, $T$ is a differential operator whose action on a function should be understood in the sense of distributions.    $u$ is thus called the weak solution of the primal variational equation \eqref{weak}.
We restrict $v\in V^*\subset L$. From \eqref{bterm},
\begin{equation*}
    \begin{split}
 (T{u}, {v})_L=&({u}, \tilde{T}{v})_L+  \langle B{u}, {v}\rangle_{W'\times W}\\
       =&({u}, \tilde{T}{v})_L+  \langle \frac{B-M}{2}u, {v}\rangle_{W'\times W}+  \langle \frac{B+M}{2} u, {v}\rangle_{W'\times W}\\
        =&({u}, \tilde{T}{v})_L+  \langle  u, \frac{B+M^*}{2}  {v}\rangle_{W'\times W}=({u}, \tilde{T}{v})_L,
    \end{split}
\end{equation*}
where we used  $u\in V=\mathcal{N}(B-M)$ and $v\in V^*=\mathcal{N}(B+M^*)$. This, combined with \eqref{weak}, gives
\begin{equation}
\label{weak-form}
({u}, \tilde{T}{v})_L= (f, v)_{L},\quad \forall\, v \in V^*.
\end{equation}
For $u\in V$, \eqref{weak-form}  is equivalent to \eqref{weak}. For $u\in L$   satisfying \eqref{weak-form}, $u$ is called the weak solution of the dual variational equation \eqref{weak-form}.

For $u \in V$, $v \in  V^*$, we define
\begin{equation}
\mathcal{L}(u,v) := \frac{|(u,\tilde Tv)_{L} - (f,v)_{L}|}{\| \tilde T v \|_{L}}.
\end{equation}
According to the estimate \eqref{a-coercive}, we have
\begin{equation*}
|(u,\tilde Tv)_{L} - (f,v)_{L}| \le  \|u \|_L \|\tilde Tv\|_L + \|f\|_L \|v\|_L \le \Big( \|u\|_L + \frac{1}{\mu_0} \|f\|_L\Big) \|\tilde Tv\|_L ,
\end{equation*}
where $\mu_0 $ is given in \eqref{p3}. Therefore, the functional $\mathcal{L}(u,v)$ is bounded with respect to $v \in  V^*$ for a fixed $u\in L$.

Thus we can reformulate the problem  \eqref{PDEproblem} or equivalently the problem \eqref{weak} as the following minimax problem formally:
 \begin{equation}
\label{eqnminmax}
\min_{u \in V} \max_{v \in  V^*} \mathcal{L}(u,v) := \min_{u \in V} \max_{v \in  V^*}\frac{|(u,\tilde Tv)_{L} - (f,v)_{L}|}{\| \tilde T v\|_{L}},
\end{equation}
to identify the weak solution of the primal variational equation \eqref{weak}.

\begin{theorem}
\label{key-thm}
Assume all the conditions given in Theorem \ref{key-FS} hold true. Then $u$ is the unique weak solution of the primal variational equation \eqref{weak} if and only if $u$ is the unique solution that solves the minimax problem \eqref{eqnminmax}.
\end{theorem}
\begin{proof}
On the one hand, if $u\in V$ is a weak solution of \eqref{weak}, we have from \eqref{weak-form} that $\mathcal{L}(u,v)=0$ for all $v \in V^*$. Thus, $u$ is a solution to the minimax problem \eqref{eqnminmax}.

On the other hand, if $u$ is a solution of the minimax problem \eqref{eqnminmax}, then
\begin{equation*}
\max_{v \in V^*} \mathcal{L}(u,v) =\max_{v \in V^*} \frac{|(u,\tilde Tv)_{L} - (f,v)_{L}|}{\| \tilde T v\|_{L}} = 0.
\end{equation*}
Thus, we have $\mathcal{L}(u,v) = 0 $ for all $v \in V^*$. This implies
\begin{equation*}
(u,\tilde T v)_{L} - (f,v)_{L} = 0,\quad \forall \ v \in  V^*.
\end{equation*}
Since $u$ is in $V$, the above equation gives
\begin{equation*}
(Tu-f,v)_{L} = 0,\quad \forall \ v \in V^*.
\end{equation*}
Observing that ${\mathcal D}$ belongs to $V^*$ and is dense in $L$, the above equation implies that $u$ is a weak solution of the primal variational equation \eqref{weak}.

Finally, under the conditions given in Theorem \ref{key-FS}, it is well known that the weak solution $u$ of the primal variational equation \eqref{weak} exists and is unique. This completes the proof of this theorem.


\end{proof}


Note that the above discussion and Theorem \ref{key-thm} are concerned with the weak solution of the primal variational equation \eqref{weak} with a solution $u$ being in $V$. It is also of interest to discuss the weak solution $u$ of the dual variational equation \eqref{weak-form}  with $u$ being in $L$ due to Friedrichs (cf. \cite{Friedrichs1958}). According to similar arguments for proving Theorem \ref{key-thm}, we have the following theorem.

\begin{theorem}
\label{key-thm2}
Assume all the conditions given in Theorem \ref{key-FS} hold true. Then $u$ is a weak solution of the dual variational equation \eqref{weak-form} if and only if $u$ is a solution of the following minimax problem:
\begin{equation}
\label{eqnminmax1}
\min_{u \in L} \max_{v \in  V^*} \mathcal{L}(u,v) =\min_{u \in L} \max_{v \in  V^*} \frac{|(u,\tilde Tv)_{L} - (f,v)_{L}|}{\| \tilde T v\|_{L}}.
\end{equation}
\end{theorem}

Note that the weak solution of the dual variational equation \eqref{weak-form} in $L$ may not be unique, which is also true for the minimax problem \eqref{eqnminmax1}. However, their solution is unique for Friedrichs' system mentioned in the Subsection \ref{PDE-system}, due to the equivalence between the weak solution and the strong solution (cf. \cite{Friedrichs1958}). In this case, the two problems \eqref{eqnminmax} and \eqref{eqnminmax1} are equivalent.

Theorems \ref{key-thm}-\ref{key-thm2} have covered various interesting equations in real applications. However, we would like to mention that Friedrichs learning can be extended to a more general setting, e.g., $u\in L$ but the data function $f$ in $Tu=f$ is not necessarily in $L$. Since the solution space $L$ is more generic than $V$ including solutions with discontinuity, this setting has a wide range of applications in fluid mechanics. We will show this application by a numerical example for the advection-reaction problem in Section 5. Theoretical analysis for more general cases is left as future work.

\section{Examples of PDEs and the Corresponding Minimax Formulation}\label{sec:EX}
Using the abstract framework and the minimax formulation developed in Section \ref{sec:FL}, we will derive the minimax formulations for several typical PDEs. From now on, we will denote by $(\cdot,\cdot)_{\Omega}$ the standard $L^2$ inner product, which induces the $L^2$ norm $\|\cdot\|_{\Omega}$. These notations also apply to $L^2$ smooth vector-valued functions.  For simplicity, we will focus on PDEs with homogeneous boundary conditions throughout this section.

\subsection{Advection-Reaction Equation}
The advection-reaction equation seeks $u$ such that
\begin{equation}\label{advection}
\mu u +\bm{\beta} \cdot \nabla u = f,
\end{equation}
where  $\bm{\beta}=(\beta_1, \cdots, \beta_d)^{\intercal} \in [L^{\infty}(\Omega)]^d$ , $\nabla \cdot \bm{\beta} \in L^{\infty}(\Omega)$, $\mu \in L^{\infty}(\Omega)$ and $f\in L^2(\Omega)$. Compared with \eqref{eqn}, \eqref{advection} is a Friedrichs' system by setting ${\bm{A}}_k = \beta_k$ for $k = 1,2,...,d$ and ${\bm{C}} = \mu$.

We assume there exists $\mu_0 > 0$ such that
\begin{equation}
\label{restriction}
\mu(\bm{x}) - \frac{1}{2}\nabla \cdot \bm{\beta}(\bm{x}) \ge \mu_0 > 0, \rm{\ \ a.e. \ in}\  \Omega.
\end{equation}
Thus, the full coercivity condition in \eqref{coercive} holds true. The graph space $W$ given by  \eqref{eqn:graph} is
\begin{equation*}
W = \{w \in L^2(\Omega);\; \bm{\beta} \cdot \nabla w \in L^2(\Omega)\}.
\end{equation*}
We define the inflow and outflow boundary for the advection-reaction equation \eqref{advection}:
\begin{equation}\label{bound}
\partial \Omega^{-} = \{\bm{x} \in \partial \Omega; \bm{\beta}(\bm{x}) \cdot \bm{n}(\bm{x}) < 0\}, \ \ \partial \Omega^{+} = \{\bm{x} \in \partial \Omega; \bm{\beta}(\bm{x}) \cdot \bm{n}(\bm{x})  > 0\}.
\end{equation}
To enforce boundary conditions, we choose from the physical interpretation that
\begin{equation}
 \label{coneboundary}
V   =  \{v \in W; v|_{\partial \Omega^{-}} = 0\},\quad
V^{*}  =  \{v \in W; v|_{\partial \Omega^{+}} = 0\}.
\end{equation}
In this case, it is easy to check that the conditions \eqref{pp1} and \eqref{pp2} hold true.
By \eqref{eqn},
\begin{equation*}
\tilde Tv = - \sum \limits_{i=1}^{d} \Big( \beta_i \frac{\partial v}{\partial x_i}  + \frac{\partial}{\partial x_i}\beta_iv\Big)+ {\bm{C}}^{\intercal}v = -\bm{\beta} \cdot \nabla v - (\nabla \cdot \bm{\beta}) v + \mu v.
\end{equation*}
The minimax problem is thus given as follows
\begin{equation*}
\min_{u \in V}\max_{v \in V^*}\mathcal{L}(u,v) = \min_{u \in V}\max_{v \in V^*}\frac{|\big(u,-(\bm{\beta} \cdot \nabla v + (\nabla \cdot \bm{\beta}) v) + \mu v \big)_{\Omega} - (f,v)_{\Omega}|}{\left\| \bm{\beta} \cdot \nabla v + (\nabla \cdot \bm{\beta}) v - \mu v\right\|_{\Omega}}.
\end{equation*}
Note that if the coercivity condition \eqref{restriction} does not hold true, we can introduce a transformation $u=e^{\lambda_0 t}\tilde{u}$, so that the advection-reaction equation \eqref{advection} in $\tilde{u}$ satisfies \eqref{restriction} for sufficiently large constant $\lambda_0>0$.

\subsection{Scalar Elliptic PDEs}
\label{sec:elliptic}
 Consider the second-order PDE  to find $u$ satisfying
\begin{equation}
\label{eqn:elliptic}
-\Delta u + \mu u =f, \quad \text{in}\ \Omega,
\end{equation}
where $\Omega\subset \mathbb{R}^d$, $\mu$ $\in L^{\infty}(\Omega)$ is positive and uniformly bounded away from zero, $f \in L^2(\Omega)$. This PDE can be rewritten into a first-order PDE system  by introducing an auxiliary function $\bm v$; i.e.,
\begin{equation*}
\bm v + \nabla u =0, \qquad
\mu u + \nabla \cdot \bm v = f.
\end{equation*}
This first order system could be formulated into a Friedrichs' system with $r = d+1$. The Hilbert space $L$ is chosen as $L = [L^2(\Omega)] ^r$. Let $\tilde {\bm u} = (\bm v^{\intercal},u)^{\intercal} \in L$. For $ k =1,2,...,d$,
$
{\bm{A}}_k =
\left[ {\begin{array}{*{20}{c}}
{\bm 0}&{\bm e^k}\\
{(\bm e^k)^{\intercal}}&{\bm 0}
\end{array}} \right], \ \
{\bm{C}} = \left[ {\begin{array}{*{20}{c}}
{{\bm{I}}_d}&{\bm 0}\\
{\bm 0}&{\mu}
\end{array}} \right],
$
where $\bm e^k$ is the $k$-th canonical basis of $\mathbb{R}^d$. Since $\mu > 0$ and has a lower bound away from zero, the full coercivity condition \eqref{coercive} is satisfied. The graph space is
\begin{equation*}
W = H({\rm div}; \Omega) \times H^1(\Omega).
\end{equation*}
One possible choice of the  Dirichlet boundary condition is as follows
\begin{equation}
V = V^* = H({\rm{div}}; \Omega) \times H_0^1(\Omega) = \{(\bm v^{\intercal},u)^{\intercal}\in W;\; u|_{\partial \Omega } = 0\}.
\end{equation}
The choices of boundary conditions are not unique, obviously. By introducing auxiliary variables, the second-order linear PDE can be reformulated into a first-order PDE system. Finally, the weak solution of \eqref{eqn:elliptic} can be found by solving the equivalent minimax problem in \eqref{eqnminmax}.

Denote the test function by $\boldsymbol{\psi} = (\psi_{\bm v}^{
\intercal}, \psi_u)^{\intercal}$ in the space $V^*$. The minimax problem can be presented as
\begin{equation*}
\label{lossellip}
\min_{ \tilde {\bm u} \in V}\max_{\boldsymbol{\psi} \in V^*}\mathcal{L}(\tilde {\bm u},\boldsymbol{\psi}) = \min_{ \tilde {\bm u} \in V}\max_{\boldsymbol{\psi} \in V^*}\frac{|\big(-\bm v, \psi_{\bm v} - \nabla \psi_u \big)_{\Omega} + (u,\mu\psi_u -\nabla \cdot \psi_{\bm v})_{\Omega}- (f,\psi_u)_{\Omega}|}{\left\|((\psi_{\bm v}-\nabla \psi_u)^{\intercal}, \mu \psi_u - \nabla \cdot \psi_{\bm v})^{\intercal} \right\|_{\Omega}}.
\end{equation*}
To reduce the computational cost, we will reformulate the above formulation into a minimax problem in a primal form. To this end, letting $\psi_{\bm v}=\nabla \psi_u$, and noting that $\tilde {\bm u}=((\nabla u)^{\intercal}, u)^{\intercal}$, we have by a direct manipulation that    
\[
\mathcal{L}(\tilde {\bm u},\boldsymbol{\psi})=\frac{| (u, \mu\psi_u -\Delta \psi_u)_{\Omega}- (f,\psi_u)_{\Omega}|}{\left\|\mu\psi_u-\Delta \psi_u\right\|_{\Omega}},
\]
which induces the following minimax problem
\begin{equation}
\label{loss:elliptic}
\min_{u\in H_0^1(\Omega)}\max_{\psi_u\in H_0^1(\Omega)} \mathcal{L}(u,\psi_u)=\min_{u\in H_0^1(\Omega)}\max_{\psi_u\in H_0^1(\Omega)}\frac{| (u, \mu\psi_u -\Delta \psi_u)_{\Omega}- (f,\psi_u)_{\Omega}|}{\left\|\mu\psi_u-\Delta \psi_u\right\|_{\Omega}}.
\end{equation}
In fact, we can derive the above minimax problem in a rigorous way.
From \eqref{eqn:elliptic}, we have
$$(-\Delta u + \mu u, \psi_u)_{\Omega} =(f, \psi_u)_{\Omega}, \qquad \forall\, \psi_u\in H_0^1(\Omega),$$
which, from the usual integration by parts twice, gives
$$(u, \mu\psi_u -\Delta \psi_u)_{\Omega}= (f,\psi_u)_{\Omega}, \qquad \forall\, \psi_u\in H_0^1(\Omega).$$ This will naturally give the minimax problem \eqref{loss:elliptic}.

\subsection{Maxwell's Equation in the Diffusion Regime}
The Maxwell's equations in $\mathbb{R}^3$ in the diffusive regime could be considered as
\begin{equation}
\label{eqn:maxwell}
\mu \bm{H} + \nabla \times \bm{E} = \bm{f},\qquad
\sigma \bm{E} - \nabla \times \bm{H} =\bm{g},
\end{equation}
with $\mu$ and $\sigma$ being two positive functions in $L^{\infty}(\Omega)$ and uniformly bounded away from zero. Three-dimensional functions $\bm{f}  ,\bm{g}$ lie in the space $[L^2(\Omega)]^3$ and the solution functions $(\bm{H}^{\intercal},\bm{E}^{\intercal})^{\intercal}$ are in the space $[L^2(\Omega)]^3 \times [L^2(\Omega)]^3$. In Equation \eqref{eqn}, set $r=6$ and let ${\bm{A}}_k \in \mathbb{R}^{6\times6}$ and $\bm{C}$ be
$
{\bm{A}}_k =
\left[ {\begin{array}{*{20}{c}}
{\bm{0}}&{\bm{\mathcal{R}^k}}\\
{(\bm{\mathcal{R}^k})^{\intercal}}&{\bm{0}}
\end{array}} \right], \ \
{\bm{C}} = \left[ {\begin{array}{*{20}{c}}
{\mu \cdot \bm{I}_3}&{\bm{0}}\\
{\bm{0}}&{\sigma \cdot \bm{I_3}}
\end{array}} \right],
$
for $k =1,2,3$. Here, the entries of $\mathcal{R}_{ij}^k = \rm{sign}(i-j)$  if $i = k+1(\rm{mod}\ 3)$ and  $\mathcal{R}_{ij}^k =0$ otherwise. The graph space is defined as
$
W = H({\rm{curl}}; \Omega) \times H({\rm{curl}}; \Omega).
$
One example of the boundary condition is
$
V = V^* =  H({\rm{curl}}; \Omega) \times H_0({\rm{curl}}; \Omega).
$
The function pair $\bm{u} := (\bm{H}^{\intercal},\bm{E}^{\intercal})^{\intercal} \in W$ is in $V$ whenever $\bm{E} \times {\bm{n}}|_{\partial \Omega} = 0$. Let $\bm{\psi} = (\bm{\psi_H}^{\intercal}, \bm{\psi_E}^{\intercal})^{\intercal}$ be the test function in $V^* $. Then the minimax problem in \eqref{eqnminmax} becomes
\begin{equation}
\label{__lossmaxwell}
\begin{split}
 \min_{\bm u\in V}\max_{\bm{\psi}\in V^*}  \frac{|\big(\bm{H}, - \nabla \times \bm{\psi_E} + \mu \bm{\psi_H} \big)_{\Omega} + \big(\bm{E},  \nabla \times \bm{\psi_H} + \sigma \bm{\psi_E} \big)_{\Omega}- (\bm{f},\bm{\psi_H})_{\Omega} - (\bm{g},\bm{\psi_E})_{\Omega} |}{\left\|((-\nabla \times \bm{\psi_E} +\mu \bm{\psi_H})^{\intercal}, (\nabla \times \bm{\psi_H} +\sigma \bm{\psi_E})^{\intercal})^{\intercal}\right\|_{\Omega}}.
\end{split}
\end{equation}

\section{Deep Learning-Based Solver}\label{sec:DL}
To complete the introduction of Friedrichs learning, we introduce a deep learning-based method to solve the minimax optimization in \eqref{eqnminmax} or \eqref{eqnminmax1} for the weak solution of \eqref{PDEproblem} or \eqref{weak-form} in this section. For simplicity, we will focus on the minimax optimization \eqref{eqnminmax} to identify the weak solution of \eqref{PDEproblem}.

\subsection{Overview}
In the deep learning-based method, one solution DNN, $\phi_s(\Bx;\theta_s)$, is applied to parametrize the weak solution $u$ in \eqref{eqnminmax} and another test DNN, $\phi_t(\Bx;\theta_t)$, is used to parametrize the test function $\psi$ in \eqref{eqnminmax}. Here, $\theta_s$ and $\theta_t$ are the parameters to be identified such that

\begin{equation}
\begin{split}
\label{_minmaxe}
(\bar \theta_s,\bar\theta_t) &= \arg \min_{\theta_s}\max_{\theta_t} L(\phi_s(\Bx;\theta_s),\phi_t(\Bx;\theta_t) )\\&= \arg \min_{\theta_s}\max_{\theta_t} \frac{|(\phi_s(\Bx;\theta_s),\tilde T \phi_t(\Bx;\theta_t) )_{\Omega} - (f,\phi_t(\Bx;\theta_t) )_{\Omega}|}{\| \tilde T \phi_t(\Bx;\theta_t)\|_{\Omega}},
\end{split}
\end{equation}
under the constraints
\begin{equation*}
\phi_s(\Bx;\theta_s) \in V\text{ and }  \phi_t(\Bx;\theta_t) \in V^*.
\end{equation*}
For simplicity, we use $L(\theta_s,\theta_t)$ for short to represent $L(\phi_s(\Bx;\theta_s),\phi_t(\Bx;\theta_t) )$ from now on.

\subsection{Network Implementation and Approximation Theory}\label{Sec_architecture}
Now, we will introduce the network structures of the solution DNN and test DNN used in the previous section. In this paper, all DNNs  are chosen as ResNet \cite{HeZhangRenSun2016} defined as follows. Let $\phi(\Bx;\theta)$ denote such a network with an input $\Bx$ and parameter  $\theta$, which is defined recursively using a nonlinear activation function $\sigma$ as follows:
 \begin{eqnarray}
 \bm{h}_0=\bm{V}\bm{x},
 \bm{g}_\ell=\sigma(\bm{W}_\ell\bm{h}_{\ell-1}+\bm{b}_{\ell}),
 \bm{h}_\ell=\bm{\bar{U}}_\ell \bm{h}_{\ell-2}+\bm{U}_\ell\bm{g}_\ell, \ell=1,2,\dots,L,
\phi(\bm{x};
 \theta)=\bm{a}^{\intercal}\bm{h}_L,
 \label{10}
 \end{eqnarray}
 where $\bm{V}\in \mathbb{R}^{m\times d}$, $\bm{W}_\ell\in \mathbb{R}^{m\times m}$, $\bm{\bar{U}}_\ell\in \mathbb{R}^{m\times m}$, $\bm{U}_\ell\in \mathbb{R}^{m\times m}$, $\bm{b}_\ell\in \mathbb{R}^{m}$ for $\ell=1,\dots,L$, $\bm{a}\in \mathbb{R}^{m}$, $\bm{h}_{-1}=\bm{0}$. Throughout this paper, $\bm{U}_\ell$ is set as an identity matrix in the numerical implementation of ResNets for the purpose of simplicity. Furthermore, as used in \cite{E2018}, we set $\bm{\bar{U}}_\ell $ as the identity matrix when $\ell$ is even and set $\bm{\bar{U}}_\ell =\bm{0}$ when $\ell$ is odd, i.e., each ResNet block has two layers of activation functions. $\theta$ consists of all the weights and biases $\{\BW^l,\Bb^l\}_{l=0}^L$. The number $m$ and $L$ are called  the width and the depth of the network, respectively. The activation function $\sigma$ is problem-dependent. For example, if the DNN as a test function is required to be continuously differentiable, the Tanh activation function can be chosen to guarantee that our DNN is in $\mathcal{C}^{\infty}$; if it is desired that $\phi(\Bx;\theta) $ is in the $H^1$ space, the activation function ReLU$(x)$ could be used, where ReLU$(x):=\max\{0,x\}$.

ResNets contain fully connected neural networks (FNNs) as special examples when $\bm{\bar{U}}_\ell =\bm{0}$ and {$U_\ell$} is the identity matrix for all $\ell$. Here, we quote existing approximation theory to briefly justify the application of neural networks as a parametrization tool in this paper. Of particular interest here is the approximation theory for Sobolev spaces $W^{n,p}$ \cite{2020Ingo,2021Ingo,hon2022simultaneous} for numerical PDEs. The following lemma is proved in \cite{2021Ingo} to describe the approximation power of neural networks quantitatively.

\begin{lemma}[Theorem 4.9 of \cite{2021Ingo}]
\label{thm:app1}
Let $d\in\mathbb{N}$, $k\in\mathbb{N}_0$, $n\in\mathbb{N}_{\geq k+1}$, and $1\leq p\leq \infty$.
There exist constants $L$, $C$, and $\tilde{\epsilon}$ such that, for every $\epsilon\in(0,\tilde{\epsilon})$ and every $f\in \{  f\in W^{n,p}((0,1)^d):\|f\|_{W^{n,p}((0,1)^d)}\leq 1\}$, there exist a FNN $\phi$ with at most $L$ layers and nonzero weights at most
\begin{equation}
    M=  \begin{cases}
                        C\epsilon^{-d/(n-k)},\qquad \text{$\max\{0,x\}^a$ activation function,} \\
                        C\epsilon^{-d/(n-k-1)}, \quad\text{\rm{Tanh} activation function,}
                    \end{cases}
\end{equation}
such that
\[
\|\phi-f\|_{W^{k,p}((0,1)^d)}\leq \epsilon.
\]
\end{lemma}

{The approximation theory in Lemma \ref{thm:app1} justifies the application of Tanh, ReLU, and the power of ReLU as activation functions in FNNs to approximate target functions in Friedrichs learning. Since ResNets of depth $L$ and width $m$ contain FNNs of depth $L$ and width $m$ as special cases, Lemma \ref{thm:app1} can also provide a lower bound of the approximation capacity of ResNets to justify the application of ResNets in our numerical examples. Lemma \ref{thm:app1} is asymptotic in the sense that it requires sufficiently large network width and depth. For quantitative results in terms of a finite width and depth, the reader is referred to \cite{hon2022simultaneous}.}

{In theory, the target function space of neural network approximation in Friedrichs learning may be as large as the $L^p$ space, which is not covered by Lemma \ref{thm:app1}. Recently, the approximation capacity of neural networks for $L^p$ spaces has been characterized in \cite{shen2021deep}. } 


\subsection{Unconstrained Minimax Problem}
When the domain becomes relatively complex,  the penalty method may be employed to solve the constrained minimax optimization in \eqref{_minmaxe}. For this purpose, we shall introduce a distance to quantify how good the solution DNN is and test how DNN satisfies its constraints. Such a distance is specified according to the boundary conditions. Denote by ${\rm{dist}}(\phi(\Bx;\theta) ,V)$ the distance between a DNN $\phi(\Bx;\theta)$ and a space $V$. Therefore, the penalty terms  of boundary conditions can be written as
{
\begin{equation}
\label{loss:boundary_penalty}
L_b(\theta_s,\theta_t):= \lambda_1 {\rm{dist}}(\phi_s(\Bx;\theta_s) ,V) + \lambda_2 {\rm{dist}}(\phi_t(\Bx;\theta_t) ,V^*),
\end{equation}}
where $\lambda_1$ and $\lambda_2$ are two positive hyper-parameters.
Finally, the {constraint minimax problem \eqref{_minmaxe} can be formulated into the} following unconstrained minimax problem
\begin{equation}\label{eqn:loss}
(\bar{\theta}_s,\bar{\theta}_t)=\arg \min_{\theta_s}\max_{\theta_t} \big( L(\theta_s,\theta_t) +L_b(\theta_s,\theta_t)\big),
\end{equation}
which can be solved to obtain the solution DNN $\phi_s(\Bx;\bar{\theta}_s)$ as the weak solution of the given PDE in \eqref{PDEproblem} by Friedrichs Learning.
\subsection{Special Networks for Different Boundary Conditions}
 \label{sec_DB}

As discussed in \cite{gu2021selectnet,gu2021structure}, it is possible to build special networks to satisfy various boundary conditions automatically, which can simplify the unconstrained optimization \eqref{eqn:loss} into
\begin{equation}\label{eqn:loss1}
(\bar{\theta}_s,\bar{\theta}_t)=\arg \min_{\theta_s}\max_{\theta_t}  L(\theta_s,\theta_t).
\end{equation}
This optimization problem \eqref{eqn:loss1} is easier to solve compared to \eqref{eqn:loss} since two hyperparameters $\lambda_1$ and $\lambda_2$ in \eqref{loss:boundary_penalty} are dropped. Note that for a regular PDE domain, e.g.,  a hypercube or a ball, it is simple to construct such special networks satisfying  various boundary conditions automatically.

Let us take the case of a homogeneous Dirichlet boundary condition   as an example. For other cases, the readers are  referred to \cite{gu2021selectnet,gu2021structure}.
A DNN satisfying the Dirichlet boundary condition
$\psi(\Bx)=g(\Bx)$ on $\partial\Omega$ can be constructed by
$
\phi(\Bx;\theta) = h(\Bx)\hat{\phi}(\Bx;\theta)+b(\Bx),
$
where $\hat{\phi}$ is a generic network as in \eqref{10}, and $h(\Bx)$ is a specifically chosen function such that $h(\Bx)=0$ on $\partial\Omega$,    and $b(\Bx)$ is chosen such that $b(\Bx)=g$ on $\partial\Omega$.
For example, if $\Omega$ is a $d$-dimensional unit ball, then $\phi(\Bx;\theta)$ can take the form
$
\phi(\Bx;\theta) = (|\Bx|^2-1)\hat{\phi}(\Bx;\theta)+b(\Bx).
$
For another example, if $\Omega$ is the $d$-dimensional hyper-cube $[-1,1]^d$, then $\phi(\Bx;\theta)$ can take the form
$
\phi(\Bx;\theta) = \underset{i=1}{\overset{d}{\prod}}(x_i^2-1)\hat{\phi}(\Bx;\theta)+b(\Bx).
$

\subsection{Network Training}
Once the solution DNN and test DNN  have been set up, the rest is to train them to solve the minimax problem in \eqref{eqn:loss}. The stochastic gradient descent (SGD) method or its variants (e.g., RMSProp \cite{Hinton2012} and Adam \cite{Kingma2014}) is an efficient tool to solve this problem numerically. Although the convergence of SGD for the minimax problem is still an active research topic \cite{Hassan2018,Daskalakis2018,Srinivasa2017}, empirical success shows that SGD can provide a good approximate solution. The training algorithm and main numerical setup are summarized in Algorithm \ref{Alg}.

In Algorithm \ref{Alg}, the outer iteration loop takes $n$ iterations. Each  inner iteration loop contains $n_s$ steps of $\theta_s$ updates and $n_t$ steps of $\theta_t$ updates.
In each inner iteration for updating $\theta_s$, we generate two new sets of random samples $\{\Bx^1_i\}_{i=1}^{N}\subset \Omega$ and $\{\Bx^2_i\}_{i=1}^{N_b}\subset\partial\Omega$ following uniform distributions. In most of the examples, the Latin Hyper-cube Sampling method is employed to generate random points in order to simulate the distributional characteristics even for the relatively small  number of samples. We define the empirical loss of these training points for the Friedrichs' system \eqref{eqn}  as
\begin{eqnarray}
L_t(\theta_s,\theta_t) := \hat{L}(\theta_s,\theta_t) +\hat{L}_b(\theta_s,\theta_t),
\label{__loss}
\end{eqnarray}
where
$
\hat{L}(\theta_s,\theta_t):=\frac{|\hat{L}_n(\theta_s,\theta_t)|}{\hat{L}_d(\theta_s,\theta_t)}
$
with

\begin{eqnarray*}
\hat{L}_n(\theta_s,\theta_t)&=  & \frac{A(\Omega)}{N_1}\sum_{i=1}^{N_1} \big(\sum_{j=1}^d \frac{\partial}{\partial x_j}(-{\bm{A}}_j  \phi_t(\Bx_i^1;\theta_t)), \phi_s(\Bx_i^1;\theta_s)\big) +\frac{A(\Omega)}{N_1}\sum_{i=1}^{N_1}\big({\bm{C}}^{\intercal}\phi_t(\Bx_i^1;\theta_t),\phi_s(\Bx_i^1;\theta_s)\big)\\
- &\frac{A(\Omega)}{N_1} &\sum_{i=1}^{N_1} \big(f(\Bx_i^1), \phi_t(\Bx_i^1;\theta_t)\big)  + \frac{A(\partial \Omega)}{N_2} \sum_{i=1}^{N_2} \big((\sum_{j=1}^{d} {\bm{A}}_j n_j) \phi_s(\Bx_i^2;\theta_s),\phi_t(\Bx_i^2;\theta_t)\big),
\end{eqnarray*}
\[
\hat{L}_d(\theta_s,\theta_t)=  \frac{A(\Omega)}{N_1}\sum_{i=1}^{N_1} \|\sum_{j=1}^d \frac{\partial}{\partial x_j}(-{\bm{A}}_j  \phi_t(\Bx_i^1;\theta_t))  + {\bm{C^{\intercal}}}\phi_t(\Bx_i^1;\theta_t)\|^2_2,
\]
where $(\cdot,\cdot)$ denotes the inner product of two vectors, $\|\cdot\|_2$ denotes the $2$-norm of vectors, $A(\cdot)$ is denoted as the area or volume of the integral region, $\frac{\partial}{\partial x_j}$ denotes the partial derivative with respect to the $j$-th argument of a function in $\Bx$, and $\{\bm{A}_j\}_{j=1}^d$ has been introduced in Section \ref{PDE-system}. As for the boundary loss, let us take the Dirichlet boundary condition $u(\Bx) = g_d(\Bx)$ as an example. In this case, the boundary loss can be formulated as

$$\hat{L}_b(\theta_s,\theta_t):=\frac{A(\partial \Omega)}{N_2} \sum_{i=1}^{N_2} \|\phi_s(\Bx_i^2,\theta_s)-g_d(\Bx_i^2)\|_2^2.$$

As mentioned in Section \ref{sec_DB}, if the solution DNN and test DNN are both built to satisfy their boundary conditions automatically, $\hat{L}_b(\theta_s,\theta_t)$ is zero.

Next, we compute the gradient of $L_t(\theta_s,\theta_t)$ with respect to $\theta_s$, denoted by $g_s$, which is known as the gradient descent direction.
The gradient is evaluated via the autograd in PyTorch, which is essentially computed by processing a sequence of chain rules since the loss function is the composition of several simple functions with explicit formulas. {For specific classes of PDEs, the computational cost of gradients can be reduced via recent development \cite{Beck2021_2}}. Besides, optimizers will use $g_s$ together with some historical gradient information to output a real descent direction, say $\tilde g_s$. Thus, $\theta_s$ can update along the direction $\tilde g_s$ as $\theta_s\leftarrow\theta_s-\eta_s \tilde g_s$. In each outer iteration of Algorithm \ref{Alg}, we repeatedly sample new training points and update $\theta_s$ for $n_s$ steps.

In each inner iteration, $\theta_t$ can be updated similarly to maximize the empirical loss $L_t(\theta_s,\theta_t)$. In each inner iteration for updating $\theta_t$, we generate random samples and evaluate the gradient of the empirical loss with respect to $\theta_t$, denoted by $g_t$. Similar to the update of $\theta_s$, $\theta_t$ can be updated via one step of ascent with a step size $\eta_t$ as follows:
$\theta_t\leftarrow\theta_t+\eta_t \tilde g_t.
$
In each outer iteration, we repeatedly sample new training points and update $\theta_t$ for $n_t$ steps.

We would like to emphasize that minimax optimization problems are generally more challenging to solve than minimization problems arising in network-based PDE solvers in the strong form. {Note that, when we fix the test DNN $\phi_t(\Bx;\theta_t)$, the loss function in \eqref{_minmaxe} is a convex functional with respect to the solution DNN $\phi_s(\Bx;\theta_s)$, but not with respect to the parameters $\theta_s$ on it}. Hence, the difficulty of the minimization problem when the test DNN is fixed is the same as the network-based least squares method. An appropriate choice of  step size is crucial to improve the solution. Moreover, in the extra step of updating test function DNN for a fixed solution DNN, the maximization problem over the test DNN is not convex {neither in the parameter space nor in the DNN space}, which {makes} the optimization even difficult.

To further facilitate the convergence of Friedrichs learning, a restarting strategy is employed to obtain the restarted Friedrichs learning in Example \ref{subsec: discontinous}, which is in the same spirit as typical restarted iterative solvers in numerical linear algebra, e.g., the restarted GMRES \cite{Joubert1994}, or the restart strategies in optimization \cite{Al-Dujaili2019,Hanada2019,Nouiehed2019,Hu2009RandomRI}. For simplicity and without loss of generality, the restarted Friedrichs learning is introduced for PDEs with Dirichlet boundary conditions. For other boundary conditions, the restarted Friedrichs learning can be designed similarly. {We stress the fact that except for the example in \ref{subsec: discontinous}, the Friedrichs learning algorithm performs well enough without a restarting strategy, so we do not implement the restarting method in the subsequent experiments.}

\begin{algorithm}
\caption{Restarted Friedrichs Learning for Weak Solutions of PDEs. }
\label{Alg}
\begin{algorithmic}
\REQUIRE The desired PDE.
\ENSURE Parameters $\theta_t$ and $\theta_s$ solving the minimax problem in \eqref{eqn:loss}.
\STATE Set iteration parameters $n$, $n_s$, and $n_t$. Set sample size parameters $N_1$ and $N_2$. Set step sizes $\eta_s^{(k)}$ and $\eta_t^{(k)}$ in the $k$-th outer iteration. Set the restart index set $\Theta_s$ and $\Theta_t$.
\STATE Initialize $\phi_s(\Bx;\theta_s^{0,0})$ and $\phi_t(\Bx;\theta_t^{0,0})$.
\FOR{$k=1,\cdots,n$}
\IF{$k\in\Theta_s$}
\STATE Keep a copy $b(\Bx) = \phi_s(\Bx,{\theta}_s^{k-1,0})$ and randomly re-initialized $\theta_s^{k-1,0}$.
\IF{the penalty method for boundary conditions is used}
\STATE Set a new DNN $\phi_s(\Bx,\theta_s^{k-1,0})=\hat{\phi}_s(\Bx,\theta_s^{k-1,0})+b(\Bx)$ with a generic DNN $\hat{\phi}_s(\Bx,\theta_s^{k-1,0})$.
\ELSE
\STATE Set a new DNN $\phi_s(\Bx,\theta_s^{k-1,0})=h(\Bx)\hat{\phi}_s(\Bx,\theta_s^{k-1,0})+b(\Bx)$ with a generic DNN $\hat{\phi}_s(\Bx,\theta_s^{k-1,0})$ and $h(\Bx)$ in \eqref{eqn:b}.
\ENDIF
\ENDIF
\FOR{$j=1,\cdots,n_s$}
\STATE Generate uniformly distributed sample points $\{\Bx^1_i\}_{i=1}^{N_1}\subset \Omega$ and $\{\Bx^2_i\}_{i=1}^{N_2}\subset\partial\Omega$.
\STATE Compute the gradient of the loss function in \eqref{__loss} at the point $(\theta_s^{k-1,j-1},\theta_t^{k-1,0})$ with respect to $\theta_s$ and denote it as $g(\theta_s^{k-1,j-1},\theta_t^{k-1,0})$.
\STATE Update $\theta_s^{k-1,j}\leftarrow\theta_s^{k-1,j-1}-\eta_s^{(k)}g(\theta_s^{k-1,j-1},\theta_t^{k-1,0})$ with a step size $\eta_s^{(k)}$.
\ENDFOR
\STATE $\theta_s^{k,0}\leftarrow\theta_s^{k-1,n_s}$.
\STATE If $k\in\Theta_t$, re-initialize $\theta_t^{k-1,0}$ randomly.
\FOR{$j=1,\cdots,n_t$}
\STATE Generate uniformly distributed sample points $\{\Bx^1_i\}_{i=1}^{N_1}\subset \Omega$ and $\{\Bx^2_i\}_{i=1}^{N_2}\subset\partial\Omega$.
\STATE Compute the gradient of the loss function in \eqref{__loss} at $(\theta_s^{k,0},\theta_t^{k-1,j-1})$ with respect to $\theta_t$ and denote it as $g(\theta_s^{k,0},\theta_t^{k-1,j-1})$.
\STATE Update $\theta_t^{k-1,j}\leftarrow\theta_t^{k-1,j-1}+\eta_t^{(k)}g(\theta_s^{k,0},\theta_t^{k-1,j-1})$ with a step size $\eta_t^{(k)}$.
\ENDFOR
\STATE $\theta_t^{k,0}\leftarrow\theta_t^{k-1,n_t}$.
\IF{Stopping criteria is satisfied}
\STATE Return $\theta_s=\theta_s^{k,0}$ and $\theta_t=\theta_t^{k,0}$.
\ENDIF
\ENDFOR
\end{algorithmic}
\end{algorithm}

\section{Numerical Experiments}\label{sec:RS}

In this section, all hyperparameters are listed in Table \ref{Tab_parameters}. We set the solution DNN $\phi_s(\Bx,\theta_s)$ as a fully connected ResNet with ReLU activation functions, depth $7$, and width $m_s$, where $m_s$ is problem dependent. The activation of $\phi_s(\Bx,\theta_s)$ is chosen as ReLU due to its capacity to approximate functions with low regularity and its good numerical performance. The test DNN $\phi_t(\Bx,\theta_t)$ has the same structure with depth $7$ and width $m_t$. To ensure the smoothness of $\phi_t(\Bx,\theta_t)$,  we employ the Tanh activation function. The optimizers for updating $\phi_s(\Bx,\theta_s)$ and $\phi_t(\Bx,\theta_t)$ are chosen as Adam and RMSProp, respectively. All of our experiments share the same setting for network structures and optimizers.
During the pre-training phase, we always set the learning rate to be larger than the following training phase. Thereafter, to ensure an effective and stable training process, the learning rate in the optimization is updated in an exponentially decaying scheme. More precisely, at the $k$-th iteration, we set the learning rate $\eta_s^{(k)} = \eta_s^{(0)}  (\frac{1}{10}) ^{(k/\nu_s)}$ for the solution DNN, where $\eta_s^{(0)}$ is the initial learning rate and $\nu_s $ is the decaying rate. Similarly, we set $\eta_t^{(k)} = \eta_t^{(0)}  (\frac{1}{10}) ^{(k/\nu_t)}$ for test DNN. The codes for reproducing the numerical results are available at \url{https://github.com/SeiruGanki/Friedrich-Learning}.
Throughout this section, special networks satisfying boundary conditions automatically are used to avoid tuning the parameters $\lambda_1$ and $\lambda_2$ in \eqref{loss:boundary_penalty}; the inner iteration numbers are set as $n_s=1$ and $n_t =1$. The values of other parameters listed in Table \ref{Tab_parameters} will be specified later.

\begin{table}
\centering
\begin{tabular}{|c|l|}
\hline
\hline
{\bf{Notation}} & {\bf{Meaning}}  \\
\hline
 \hline
 $d$ & the dimension of the problem \\  \hline
 $n_p$ & the number of pre-training iterations  \\ \hline
 $n$ & the number of outer iterations  \\ \hline
  $\eta_s^{p}$ & the pre-training learning rate for optimizing the solution network \\ \hline
 $\eta_t^{p}$ & the pre-training learning rate for optimizing the test network\\ \hline
 $\eta_s^{(0)}$ & the initial learning rate for optimizing the solution network \\ \hline
 $\eta_t^{(0)}$ & the initial learning rate for optimizing the test network\\ \hline
 $\nu_s$ & the decaying rate for $\eta_s$\\ \hline
 $\nu_t$ & the decaying rate for $\eta_t$\\ \hline
 $m_s$ & the width of each layer in the solution network\\\hline
 $m_t$ & the width of each layer in the test network \\\hline
 $n_s$ & the number of inner iterations for the solution network \\\hline
 $n_t$ & the number of inner iterations for the test network \\\hline
 $N$ & the number of training points inside the domain \\\hline
 $N_b$ & the number of training points on the domain boundary \\\hline
 $\Theta_s$ & the restart index set of the solution network\\ \hline
$\Theta_t$ &the restart index set of the test network\\ \hline
\end{tabular}
\caption{\em Parameters in the model and algorithm.}
\label{Tab_parameters}
\end{table}
To measure the solution accuracy, the following discrete relative $L^2$ error at uniformly distributed test points in the domain is applied; i.e.,
$$
e_{L^2}(\theta_s):=\left(\frac{\underset{i}{\sum}\left\|\phi_s(\Bx_i;\theta_s)-u^*(\Bx_i)\right\|_2^2}{\underset{i}{\sum}\left\|u^*(\Bx_i)\right\|_2^2}\right)^\frac{1}{2},
$$
where $u^*$ is the exact solution. In the case when the true solution is continuous, the following discrete relative $L^\infty$ error at uniformly distributed test points in the domain is also applied; i.e.,
$$
e_{L^\infty}(\theta_s):=\frac{\max_i (\left\|\phi_s(\Bx_i;\theta_s)-u^*(\Bx_i)\right\|_{\infty})}{\max_i (\left\|u^*(\Bx_i)\right\|_{\infty})},
$$
where $\|\cdot\|_\infty$ denotes the $L^\infty$-norm of a vector. In most examples, we choose at least $10,000$ testing points for error evaluation. When the dimension is high or the value of the target function surges, we may choose $50,000$ or even $100,000$ testing points.

\subsection{Advection-Reaction Equation with Plain Discontinuity}
\label{subsec: discontinous}
In the first example, we identify the weak solution in $L^2(\Omega)$ of the advection-reaction equation in \eqref{advection} with discontinuous solutions. Following Example $2$ in \cite{Houston2000}, we choose the velocity $\bm \beta = (1,9/10)^{\intercal}$ and $\mu = 1$ in the domain $\Omega = [-1,1]^2$.
We choose the right-hand-side function $f$ and the boundary function $g$ such that the exact solution is
\begin{equation}
u^*(x,y) =
\left\{ \begin{array}{cl}
\sin(\pi (x+1)^2/4)\sin (\pi(y-\frac{9}{10}x)/2)\ & {\rm{for}} \ -1\le x \le1,\ \frac{9}{10}x<y\le 1,\\
e^{-5(x^2+(y-\frac{9}{10}x)^2)}  &{\rm{for}} \ -1\le x \le1,\ -1\le y< \frac{9}{10}x.\\
\end{array} \right.
\label{discontinuousDG}
\end{equation}
The exact solution is visualized in Figure \ref{3D_disconti_exactslu}. The discontinuity of the initial value function will propagate along the characteristic line $y=9x/10$. Hence, the derivative of the exact solution does not exist along that line. Classical network-based least square algorithms in the strong form will encounter a large residual error near the characteristic line and hence its accuracy may not be very attractive, which motivates our Friedrichs Learning in the weak form.

As discussed in \cite{Houston2000}, a priori knowledge of the characteristic line is crucial for conventional finite element methods with adaptive mesh to obtain high accuracy. In \cite{Houston2000}, the streamline diffusion method (SDFEM) can obtain a solution with $O(10^{-2})$ accuracy using $O(10^4)$ degrees of freedom when the mesh is aligned with the discontinuity, i.e., when the priori knowledge of the characteristic line is used in the mesh generation. The discontinuous Galerkin method (DGFEM) in \cite{Houston2000} can obtain $O(10^{-8})$ accuracy under the same setting. When the mesh is not aligned with the discontinuity, e.g., when the characteristic line is not used in mesh generation, DGFEM converges as slow as SDFEM and the accuracy is not better than $O(10^{-2})$ with $O(10^4)$ degrees of freedom according to the discussion in \cite{Houston2000}.

As a deep learning algorithm, Friedrichs Learning is a mesh-free method and the weak solution can be identified without the priori knowledge of the characteristic line. By the discussion in Section \ref{sec_DB}, a special network $\phi_s(\Bx,\theta_s)$ is constructed as follows to fulfill the boundary condition of the solution:
\begin{equation}
\label{discontinuous_network_autofill}
\phi_s(\Bx,\theta_s) = \cos(-\frac{\pi}{4}+\frac{\pi}{4}x)\cos(-\frac{\pi}{4}+\frac{\pi}{4}y) \hat \phi_s(\Bx,\theta_s) + b(x,y),
\end{equation}
where $b(x,y)$ is constructed directly from the boundary condition as
\begin{equation}
b(x,y) =
\left\{ \begin{array}{cl}
0,\ & {\rm{for}} \ -1\le x \le1,\ -0.4+x/2<y\le 1,\\
\begin{array}{l}
e^{-5[x^2+(-1-9/10x)^2]} + e^{-5[(-1)^2+(y+9/10)^2]} \\
\quad \quad \quad \quad \quad - e^{-5[(-1)^2+(-1+9/10)^2]},
\end{array}
&{\rm{for}} \ -1\le x \le1,\ -1\le y \le -0.4+x/2,\\
\end{array} \right.
\label{__discontinuousBASE}
\end{equation}
satisfying $b(x,y) = u(x,y)$ on the inflow boundary $\partial \Omega^-$. For test function, we fix its structure so that $\phi_t(\Bx,\theta_t) =0$ on $\partial \Omega^+$ defined in \eqref{bound}.

First of all, the restarting strategy as introduced at the end of section \ref{sec:DL} for pre-training the base function is employed. The special network structure satisfying the Dirichlet boundary conditions for solution DNN $\phi_s$ is constructed as
\begin{equation}
\label{eqn:b}
\phi_s(\Bx;\theta_s) = h(\Bx)\hat{\phi}_s(\Bx;\theta_s)+b(\Bx),
\end{equation}
where $b(\Bx)$ satisfies the boundary condition which also can be regarded as an initial guess; $h(\Bx)=0$ on the Dirichlet boundary. We observe  that if  $b(\Bx)$ is closer to the true solution, it is easier to train a generic DNN $\hat{\phi}_ s$ to obtain the solution DNN $\phi_s$ that approximates the true solution more accurately. Therefore, after a few rounds of outer iterations in the original Friedrichs learning, we obtain  a rough solution DNN, which can be served as a better $b$ function in \eqref{eqn:b} to construct a new solution DNN. After that, we will continue training to obtain a more accurate solution.

Secondly, we choose $b(x,y)$ to be discontinuous along a random line rather than the true discontinuous line of the exact solution. This could be a reasonable reproduction of the real application scenarios. Indeed, our choice of $b(x,y)$ above actually makes the problem more challenging. The true solution is discontinuous along the characteristic line, the blue line in Figure \ref{sample_dis}, and $b(x,y)$ is discontinuous along the orange line in Figure \ref{sample_dis}. Hence, to make the solution DNN $\phi_s$ in \eqref{discontinuous_network_autofill} approximate the true solution well, one algorithm needs to find and correct these two lines automatically and the DNN $\hat{\phi}_s$ in \eqref{discontinuous_network_autofill} should be approximately discontinuous along these two lines. As shown by Figure \ref{2D_disconti_diff}, with Friedrichs learning the solution DNN $\phi_s$ has a configuration similar to the true solution in Figure \ref{3D_disconti_exactslu}, which means that it has successfully learned these two lines. This feature can be significant because no prior knowledge of the discontinuity of the exact solution is needed during the training, as long as the boundary condition is satisfied.

Thirdly, we can observe the mechanism of Friedrichs learning from Figure \ref{2D_disconti_test}, where the test DNN $\phi_t$ surges and has a larger magnitude near these two lines to emphasize the error of the solution DNN $\phi_s$. It can make the update of the configuration of $\phi_s$  more focused on these two lines than other places, which in turn facilitates the expected convergence of the solution DNN.

The whole training process can be divided into two phases due to restarting. In Phase I of pre-training, we train a ResNet of width $50$ for $1,000$ outer iterations to get a rough solution with an $L^2$ relative error $2.76e\text{-}1$. All other parameters are shown in Table \ref{tab:discontinous}. As shown in Figure \ref{3D_disconti_basefunc}, the rough solution has already captured basically the shape of the solution. In Phase II of training, we set this rough solution as a base function $b(\Bx)$ and again set up a ResNet of width 150. It is shown that $50,000$ outer iterations are enough to make the $L^2$ error of the solution DNN decrease to $2.27e\text{-}2$, as shown in Figure  \ref{2D_disconti_diff} and Figure \ref{fig_loss_discontinuous_50k}. Our method is comparable with the SDFEM in \cite{Houston2000} considering the same order of degrees of freedom summarized in Table \ref{tab:discontinous}. However, SDFEM in \cite{Houston2000} requires the priori knowledge of the characteristic line while our method does not. Therefore, from the perspective of practical computation, our method would be more convenient in real applications.

 To compare Friedrichs Learning and the  DNN-based least square (LS) algorithm \cite{Dissanayake1994,Lagaris1998,RAISSI2019686}, we conduct comparative experiments with very similar hyper-parameters shown in Table \ref{tab:discontinous-LS}.  After $50,000$ iterations we obtain a solution with the relative error in $L^2$ norm which is $3.29e\text{-}2$ as shown in \ref{fig_loss_discontinuous_50k}. It is worth pointing out that the iteration shown is the outer iteration, and the computation of Friedrichs learning costs about twice as much as the LS approach for each iteration. Though Friedrichs learning is more accurate, the DNN-based least square algorithm and the Friedrichs learning have errors of the same order in this numerical test.

\begin{figure}[htbp]
    \subfigure[The characteristic line (blue) {of the exact solution} and the line (orange) along which $b(x,y)$ in \eqref{discontinuous_network_autofill} is discontinuous.]{
	\begin{minipage}[t]{0.31\textwidth}
		\centering
		\includegraphics[width=0.85\linewidth]{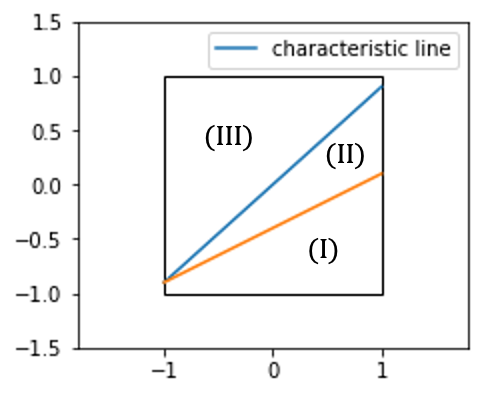}
        \label{sample_dis}
	\end{minipage}
	}
	\subfigure[Exact solution.]{
	\begin{minipage}[t]{0.31\textwidth}
		\centering
		\includegraphics[width=0.85\linewidth]{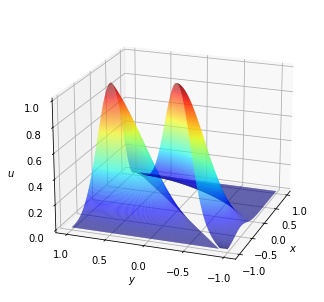}
        \label{3D_disconti_exactslu}
	\end{minipage}
	}
	\subfigure[The solution DNN right before restarting.]{
	\begin{minipage}[t]{0.31\textwidth}
		\centering
		\includegraphics[width=0.85\linewidth]{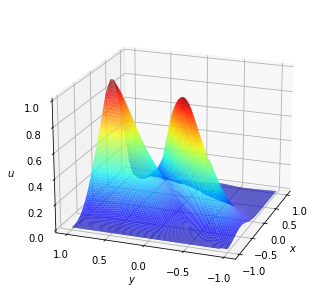}
        \label{3D_disconti_basefunc}
	\end{minipage}
	}
	\subfigure[The point-wise error of approximate solution at epoch $50,000$ by Friedrichs learning.]{
	\begin{minipage}[t]{0.31\textwidth}
		\centering
		\includegraphics[width=0.85\linewidth]{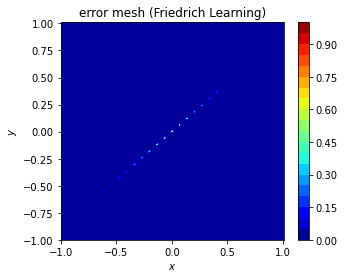}
        \label{2D_disconti_diff}
	\end{minipage}
	}
   \hfil
	\subfigure[The test DNN value at epoch $50,000$.]{
	\begin{minipage}[t]{0.31\textwidth}
		\centering
		\includegraphics[width=0.85\linewidth]{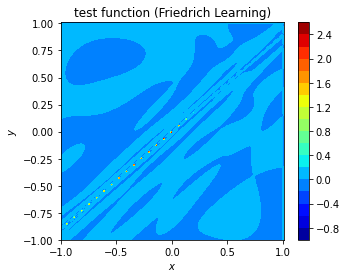}
        \label{2D_disconti_test}
	\end{minipage}
	}
    \hfil
 	\subfigure[The relative $L^2$ error curve by DNN-based least square and Friedrichs learning ]{
	\begin{minipage}[t]{0.31\textwidth}
		\centering
		\includegraphics[width=0.85\linewidth]{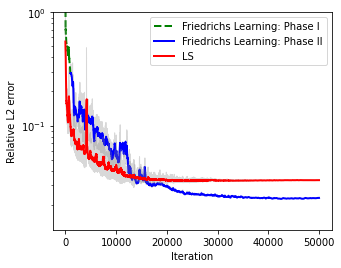}
        \label{fig_loss_discontinuous_50k}
	\end{minipage}
	}
    \hfil
\centering
\caption{Numerical results of Equation \eqref{advection} when the exact solution is chosen as \eqref{discontinuousDG}. }
\end{figure}


\begin{table}
\centering
\begin{tabular}{|c|c|c|c|c|c|c|}
  \hline
  Parameters & $n$  & $m_s$ & $m_t$ &$N$  & $N_b$ &$\Theta_s$\\ \hline
  Value & $50,000$  & pre-train $ 50$, after $250$  & $150$ &$90,000$ & $45,000$ & $\{1,000\}$\\ \hline
  Parameters  & $\eta_s^{(0)}$ &$ \eta_t^{(0)}$ & $\nu_s$ & $\nu_t$ & parameter number &$\Theta_t$\\\hline
 Value & $3e\text{-}4$ & $3e\text{-}3$ &$9,000$&$9,000$ & $327,700$ &$\varnothing$\\ \hline
\end{tabular}
\vspace{0.25cm}
\caption{\em The parameters for the Friedrichs learning solver of the experiment in Section \ref{subsec: discontinous}.}
\label{tab:discontinous}
\end{table}

\begin{table}
\centering
\begin{tabular}{|c|c|c|c|c|c|}
  \hline
  Parameters &$n$&  $m_s$ & $N$   & $\eta_s^{(0)}$ &$\nu_s$ \\ \hline
  Value &$50,000$ & $250$ & $90,000$   & $1e\text{-}3$ &$10,000$  \\ \hline
\end{tabular}
\vspace{0.25cm}
\caption{\em The parameters of the comparative experiment in Section \ref{subsec: discontinous}.}
\label{tab:discontinous-LS}
\end{table}

\subsection{Advection-Reaction Equation with Curved Discontinuity}
\label{subsec:disconti_wind} Consider a domain $\Omega=\{(x,y)|x^2+y^2\le 1,y\ge 0\}$. The velocity $\beta = (\sin \theta, -\cos \theta)^{\intercal} = (y/\sqrt{x^2+y^2},-x/\sqrt{x^2+y^2})$ with $\theta$ being the polar angle and $\mu = 0$. The Dirichlet boundary condition on the inflow boundary is given as $u(x,0)=1$ for $-1\le x \le -\frac{1}{2}$, $u(x,0)=0$ for $-\frac{1}{2}< x \le 0$. The true solution is
\begin{equation}
\label{discontinuous_wind_exact}
u^*(x,y) = \left\{  \begin{array}{l}
0, \ \ x^2+y^2 < 1/4\\
1, \ \ x^2+y^2 \ge 1/4
\end{array}\right..
\end{equation}

Again, without the prior knowledge of the characteristic line, to create a network satisfying the boundary condition, we choose a solution DNN $\phi_s$ as
\begin{equation}
\label{discontinuous_wind_network_autofill}
\phi_s(\Bx,\theta_s) = \big(\frac{\pi}{2}-\arctan(\frac{-x}{y})\big)\sin(\frac{\pi}{2}r) \hat \phi_s(\Bx,\theta_s) + b(x,y),
\end{equation}
where
$$b(x,y) = \left\{  \begin{array}{l}
0, \ \ x \ge -1/2\\
1, \ \ x < -1/2
\end{array}\right., {\rm{\ \ and}}\ \  r = \sqrt{x^2+y^2}.$$
$\phi_s$ will be applied as the solution network of Friedrichs learning. Similarly, \begin{equation}
\label{discontinuous_wind_network_autofill_test}
\phi_t(\Bx,\theta_t) = \Big(-\frac{\pi}{2}-\arctan(\frac{-x}{y})\Big)\hat \phi_t(\Bx,\theta_t).
\end{equation}
By applying Friedrichs learning with $\phi_s$ and $\phi_t$ as the solution and test DNN, respectively, we get an approximate solution with an $L^2$ relative error
$2.48e\text{-}2$ with the iteration error visualized in Figure \ref{discontinous_wind_project_FL}. Figure \ref{discontinuous_wind_error} shows the point-wise error after 100,000 iterations by Friedrichs learning.
Friedrichs learning can capture the discontinuous locations well with sharp characterization. The test function value is relatively large around the discontinuous place, resulting in a greater weight for samples around there, which can help to obtain a more accurate PDE solution. Our experiments are implemented on the graphic card Nvidia Tesla P100 with CUDA; in this example, for $10,000$ iterations it will take about 50 minutes and cost twice as much as the Least Square methods.

As a comparison with traditional PDE solvers, note that the same PDE was solved by the adaptive least-squares finite element method (LSFEM) in \cite{Liu2020} with the same order of degrees of freedom ($\approx 1.1 \times 10^5$) as in Friedrichs learning. The $L^2$ relative error of LSFEM is $4.59e\text{-}2$, which is larger than the one by Friedrichs learning. We would like to emphasize that LSFEM in \cite{Liu2020} has applied extra computational resources to adaptively generate discretization mesh, without which the error would be poorer. Besides, the DGFEM\footnote{Available at \url{https://github.com/dealii/dealii}.} with adaptive mesh is also applied to solve the same PDE with the same order of degrees of freedom (107,332) as in Friedrichs learning. The $L^2$ relative error of DGFEM is $2.05e\text{-}2$, which is very similar to the error by Friedrichs learning. Following the idea in \cite{Liu2020} to visualize the solution, we project the approximate solutions by DGFEM and Friedrichs learning to the radius axis in Figure \ref{discontinous_wind_project_FL} and plot the scatters corresponding to the angle $\theta$ ranging from $0$ to $\pi$, the points chosen is the same as DGFEM following the software built-in functions. This visualization makes it easier to compare the solutions near the discontinuous location. It is easy to see that the solution by DGFEM has a larger error than the one by Friedrichs learning near the discontinuous location.

DNN-based least square is also applied to solve the same problem as a comparison. Two options of DNN-based least square are tested: one with $\phi_s$ as the solution network so that there is no penalty terms to enforce the boundary condition in the loss function; another one with a standard neural network as the solution network and, hence, a penalty term in the loss function is added to enforce the boundary condition. The first option, i.e., DNN-based least square with the special network structure described in \eqref{discontinuous_wind_network_autofill} to parametrize the PDE solution, fails to find a reasonable solution even though the optimization loss is almost zero as shown by Figure \ref{discontinuous_wind_error_list}. One possible reason is due to the fact that the square loss in the strong form is $0$ for $b(x,y)$, since DNN-based least square samples points randomly in the ``interior" but not on the discontinuous line with probability almost 1. Therefore, even if the generic network $\hat \phi_s(\Bx,\theta_s)$ is not $0$ at the beginning, no information of the discontinuity is captured by the strong form in DNN-based least square and, hence, the solution network will converge to $0$, resulting in a fake solution satisfying the equation almost everywhere in the strong sense. However, this solution is mathematically wrong in the weak sense. For instance, the derivatives across the discontinuity contain Dirac’s delta functions.

The second option of DNN-based least square can provide a meaningful solution and serves as a good baseline for Friedrichs learning. Figure \ref{discontinuous_wind_error} shows the point-wise error after 100,000 iterations by DNN-based least square with a boundary penalty term and Friedrichs learning.
Friedrichs learning can capture the location of discontinuous line with better accuracy than DNN-based least square. The error curve of DNN-based least square in the $L^2$ norm is shown in \ref{discontinuous_wind_error_list} (the red line) and the iteration error cannot be improved anymore at the early beginning. DNN-based least square with a boundary penalty term provides a solution with an $L^2$ error $9.35e\text{-}2$ after $100,000$ iterations and this error is almost $4$ times as the error by Friedrichs learning.

\begin{figure}[htbp]
    \subfigure[Top, the point-wise error for solution by DNN-based least square; Middle, the point-wise error for solution by Friedrichs learning; Bottom, the point-wise test function value by Friedrichs learning.]{
	\begin{minipage}[t]{0.41\textwidth}
		\centering
		\includegraphics[width=0.8\linewidth]{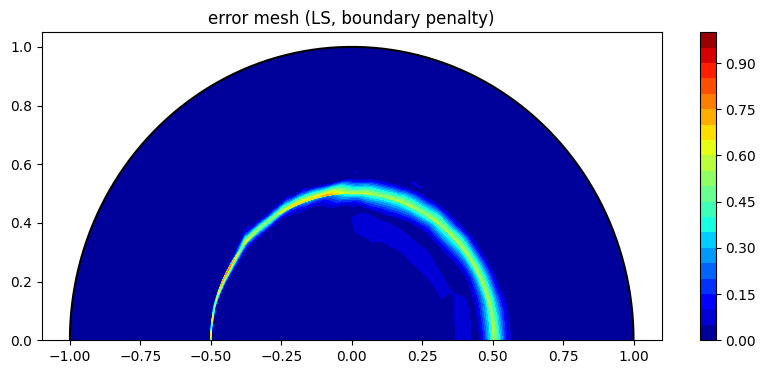}
        \includegraphics[width=0.8\linewidth]{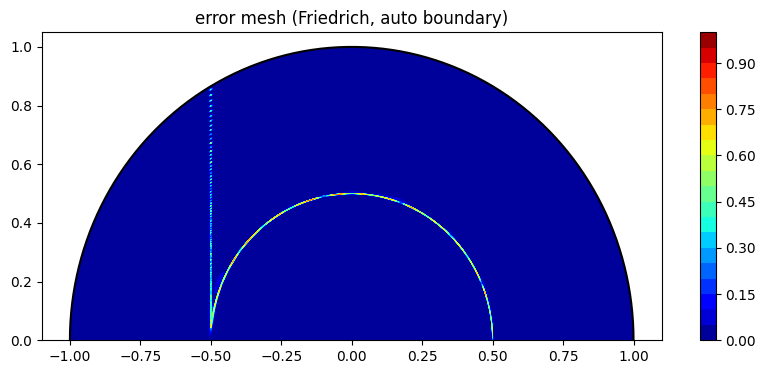}
        \includegraphics[width=0.8\linewidth]{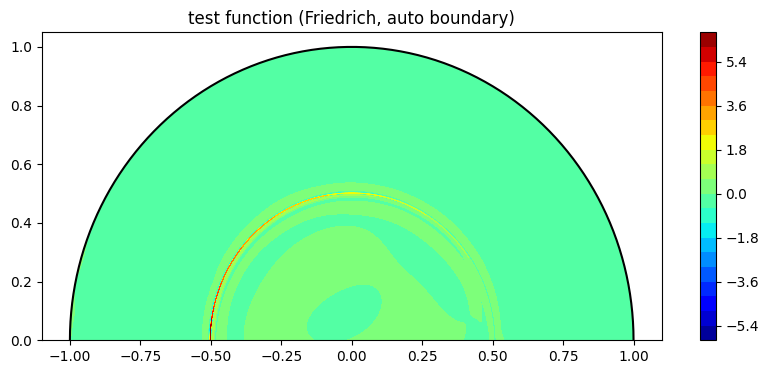}
        \vfil
        \label{discontinuous_wind_error}
	\end{minipage}
	}
    \hfil
	\subfigure[Upper, the relative $L^2$ error curve with respect to the iteration number for three different algorithm settings; Lower, the running DNN-based least square loss with respect to the iteration number.]{
	\begin{minipage}[t]{0.41\textwidth}
    \vspace{-2.3cm}
		\centering
		\includegraphics[width=0.7\linewidth]{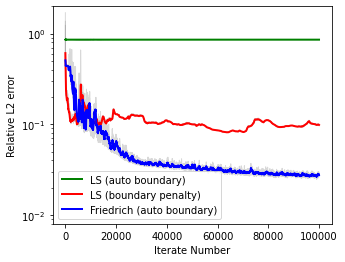}
        \includegraphics[width=0.7\linewidth]{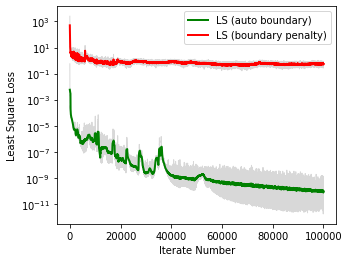}
        \vfil
        \label{discontinuous_wind_error_list}
	\end{minipage}
	}
 \hfil
\centering
\caption{Numerical results of Equation \eqref{advection} when the exact solution is chosen as \eqref{discontinuous_wind_exact}. }
\end{figure}
\begin{figure}[htbp]
    \subfigure[Projected solution of DGFEM with adaptive mesh grid.]{
	\begin{minipage}[t]{0.41\textwidth}
		\centering
		\includegraphics[width=0.8\linewidth]{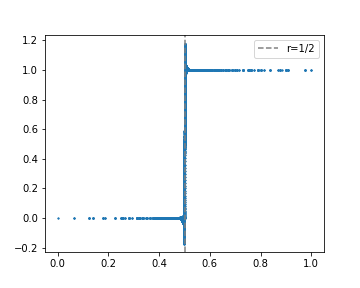}
        \vfil
        \label{discontinous_wind_project_DGFEM}
	\end{minipage}
	}
    \hfil
	\subfigure[Projected solution of Friedrichs learning. The value of points 0.005 Euclid distance away from $x=-\frac{1}{2}$ is adjusted to true value.]{
	\begin{minipage}[t]{0.41\textwidth}
    \vspace{-4.3cm}
		\centering
		\includegraphics[width=0.8\linewidth]{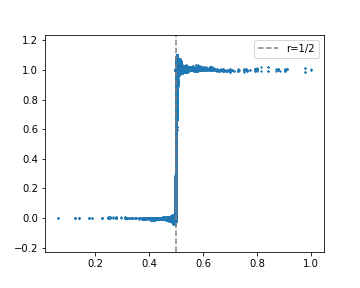}
        \vfil
        \label{discontinous_wind_project_FL}
	\end{minipage}
	}
 \hfil
\centering
\caption{Numerical results of Equation \eqref{advection} when the exact solution is chosen as \eqref{discontinuous_wind_exact}. }
\end{figure}

\begin{table}
\centering
\begin{tabular}{|c|c|c|c|c|c|c|}
\hline
  Parameters & $n$  & $m_s$ & $m_t$ &$N$  & $N_b$\\ \hline
  Value & $100,000$  & $150$ & $150$ &$45,000$ & $5,000$\\ \hline
  Parameters  & $\eta_s^{(0)}$ &$ \eta_t^{(0)}$ & $\nu_s$ & $\nu_t$ & parameter number \\\hline
 Value & $3e\text{-}4$ & $3e\text{-}3$ &$15,000$&$15,000$ & 113,850\\ \hline
\end{tabular}
\vspace{0.25cm}
\caption{\em The parameters for the Friedrichs learning solver of the experiment in Section \ref{subsec:disconti_wind}.}
\label{tab:disconti_wind}
\end{table}

\begin{table}
\centering
\begin{tabular}{|c|c|c|c|c|c|c|}
  \hline
  Parameters &$n$ & $N$   & $\eta_s^{(0)}$ &$\nu_s$ &  $m_s$ \\ \hline
  Value & $100,000$ & $45,000$   & $1e\text{-}3$ & $15,000$ & $150$ \\ \hline
\end{tabular}
\vspace{0.25cm}
\caption{\em The parameters of the comparative experiment in Section \ref{subsec:disconti_wind}.}
\label{tab:discontinous-wind-LS}
\end{table}

\subsection{Green's Function}
\label{subsec:green}
The next example is to identify the Green's function of the Laplacian operator by solving
\begin{equation}\label{eqn:gf}
\Delta u(\Bx) = \delta_0(\Bx),
\end{equation}
where $\delta_0(\Bx)$ is the Dirac's delta function at the origin. In this example, we solve the above equation on a 3D unit ball $\Omega = \{\Bx\in\mathbb{R}^3|\left\|\Bx\right\|_2 \le 1\}$. The true solution is
\begin{equation}
\label{green_true_slu}
u^*(\Bx) =  \frac{1}{8\pi\left\|\Bx\right\|_2},
\end{equation}
and the given Dirichlet boundary condition is $u(\Bx) =  \frac{1}{8\pi}$ on $\partial \Omega$. Although the exact solution is in $H^1$ and has strong singularity near the origin, Friedrichs learning can provide an approximate solution with a small error as shown in Figure \ref{green_true} and \ref{green_error_FL}. Figure \ref{green_error_FL} visualizes the point-wise relative error of the solution by Friedrichs learning. We can see that, except for those locations that are very close to the origin, the relative errors are not greater than $1e\text{-}1$. In Table \ref{tab:green_error}, we summarize the relative $L^2$ errors of the solution by Friedrichs learning in the region of $\Omega \backslash \mathcal{B}(\bm{0},\varepsilon)$ with $\varepsilon$ equal to $0.001, 0.01, 0.1,0.2$, respectively. Therefore, the solution is accurate when the location is not very close to the origin.

As a comparison, the DNN-based least square method cannot find a meaningful solution for the Green's function. The right hand side function of \eqref{eqn:gf} is a Dirac Delta function and, hence, cannot be captured by the discrete analog of the least square loss function via random sampling. Therefore, even if the DNN-based least square method can be applied to form an optimization problem, the minimizer of this problem will return a constant function as a solution, which has a large error.

\begin{figure}[htbp]
    \subfigure[The cross section of the Green's function at $x_3=0$. The Green's function has strong singularity near the origin.]{
	\begin{minipage}[t]{0.29\textwidth}
		\centering
		\includegraphics[width=0.72\linewidth]{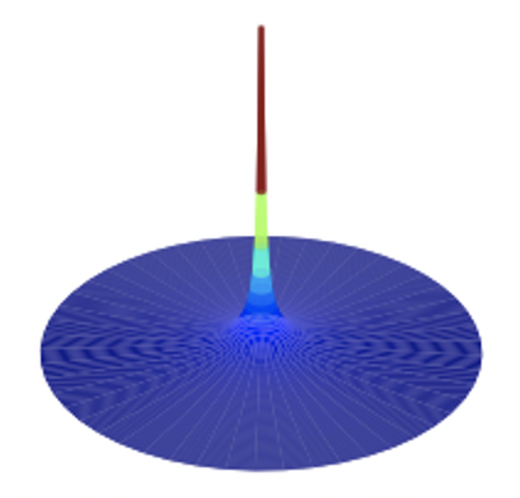}
        \vfil
        \label{green_true}
	\end{minipage}
	}
    \hfil
	\subfigure[The projected point-wise relative error by Friedrichs learning on the slice $x_3=0$.]{
	\begin{minipage}[t]{0.29\textwidth}
		\centering
		\includegraphics[width=0.85\linewidth]{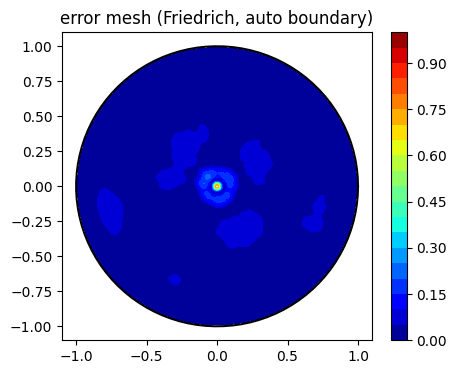}
        \vfil
        \label{green_error_FL}
	\end{minipage}
	}
      \hfil
	\subfigure[The relative $L^2$ and  maximum error curve with respect to the iteration number.]{
	\begin{minipage}[t]{0.29\textwidth}
		\centering
		\includegraphics[width=0.85\linewidth]{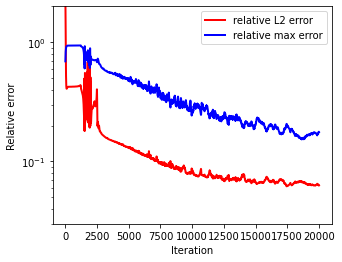}
        \vfil
        \label{green_error_list}
	\end{minipage}
    }
 \hfil
\centering
\caption{Numerical results of Equation \eqref{advection} when the exact solution is chosen as \eqref{green_true_slu}. }
\end{figure}

\begin{table}
\centering
\begin{tabular}{|c|c|c|c|c|c|c|c|}
\hline
  Parameters & $n$  & $m_s$ & $m_t$ &$N$  & $N_b$ & $\eta_s^p$\\ \hline
  Value & $20,000$  & $100$ & $100$ &$45,000$ & $5,000$& $1e\text{-}4$\\ \hline
  Parameters  &$\eta_t^p$ & $\eta_s^{(0)}$ &$ \eta_t^{(0)}$ & $\nu_s$ & $\nu_t$ & parameter number \\\hline
 Value &$2e\text{-}4$& $1e\text{-}5$ &$2e\text{-}5$&$10,000$&$10,000$ & $51,000$\\ \hline
\end{tabular}
\vspace{0.25cm}
\caption{\em The parameters for the Friedrichs learning solver of the experiment in Section \ref{subsec:green}.}
\label{tab:green}
\end{table}


\begin{table}
\centering
\begin{tabular}{|c|p{2.5cm}<{\centering} |}
  \hline
 $\varepsilon$ & mean  \\ \hline
  $0.2$ & $3.47e\text{-}2$ \\
  $0.1$  & $4.43e\text{-}2$  \\
$ 0.01$ & $8.16e\text{-}2$ \\
 $0.001$ & $9.39e\text{-}2$\\ \hline
\end{tabular}
\vspace{0.25cm}
\caption{\em The relative $L^2$ errors by Friedrichs learning in the region of $\Omega \backslash \mathcal{B}(\bm{0},\varepsilon)$ for the Green's function experiment in Section \ref{subsec:green}.}
\label{tab:green_error}
\end{table}

\subsection{High-Dimensional Advection-Reaction Equation}
\label{subsec:highdim}
We consider a 10D advection equation with discontinuity in the domain $[0,1]^{10}$. In particular, we find $u=u(\Bx)$ such that
\begin{equation}
2\Big(1+\exp\Big(-\big(\sum_{i=3}^{10}x_i\big)^2\Big)\Big)u_{x_1} + \exp(2x_1)u_{x_2} = 0,
\end{equation}
 where $u_{x_1}=\frac{\partial u}{\partial x_1}$ and $u_{x_2}=\frac{\partial u}{\partial x_2}$. The exact solution is  \begin{equation}
\label{highdim_exact}
 u^*(\Bx) = g\Big(\exp(2x_1)-4\Big(1+\exp\Big(-\big(\sum_{i=3}^{10}x_i\big)^2\Big)\Big)x_2\Big),
\end{equation}
where \begin{equation*}
g(x) = \left\{
\begin{array}{l}
1, \ \ x >0 \\
0, \ \ x\le0
\end{array}
\right. .
\end{equation*}
The Dirichlet boundary condition is given on the inflow boundary $\{\Bx| x_1=0 \ {\rm or } \  x_2 =0\}$.
\begin{figure}[htbp]
    \subfigure[The projected point-wise error by Friedrichs learning on the slice $x_i=\frac{1}{2}, i=3,4,\dots,10$.]{
	\begin{minipage}[t]{0.45\textwidth}
		\centering
		\includegraphics[width=0.8\linewidth]{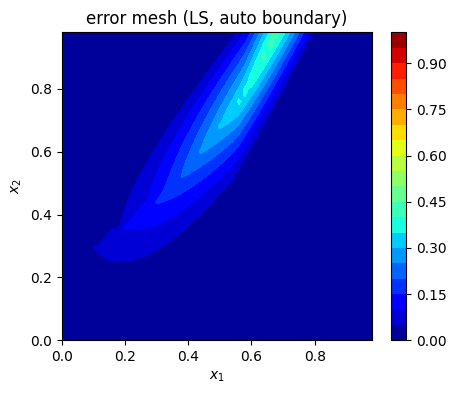}
        \vfil
        \label{highdim_error_PINN}
	\end{minipage}
	}
    \hfil
	\subfigure[The projected point-wise error by DNN-based least square on the slice $x_i=\frac{1}{2}, i=3,4,\dots,10$.]{
	\begin{minipage}[t]{0.45\textwidth}
		\centering
		\includegraphics[width=0.8\linewidth]{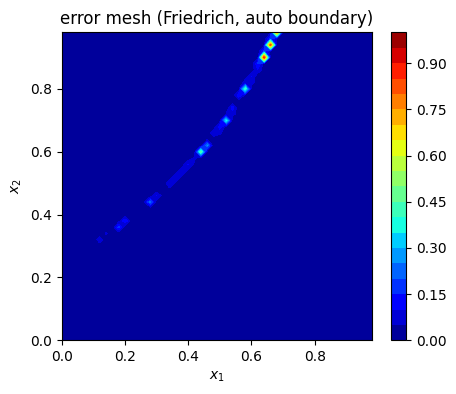}
        \vfil
        \label{highdim_error_FL}
	\end{minipage}
	}
 \hfil
 \subfigure[The projected point-wise test function value on the slice $x_i=\frac{1}{2}, i=3,4,\dots,10$.]{
	\begin{minipage}[t]{0.45\textwidth}
		\centering
		\includegraphics[width=0.8\linewidth]{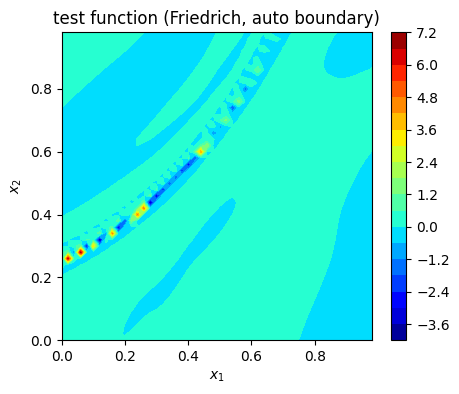}
        \vfil
        \label{highdim_test}
	\end{minipage}
	}
 \hfil
 \subfigure[The relative $L^2$ error curve with respect to the iteration number by DNN-based least square and Friedrichs learning.]{
	\begin{minipage}[t]{0.45\textwidth}
		\centering
		\includegraphics[width=0.75\linewidth]{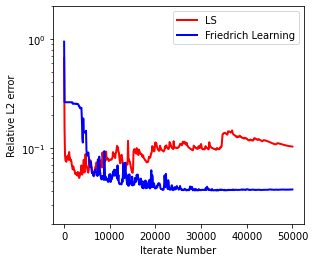}
        \label{highdim_errorlist}
	\end{minipage}
	}
 \hfil
\centering
\caption{Numerical results of Equation \eqref{advection} when the exact solution is chosen as \eqref{highdim_exact}. }
\end{figure}

Figure  \ref{highdim_error_PINN} and Figure \ref{highdim_test} show that Friedrichs learning can identify the location of low regularization by test DNNs in this high-dimensional problem. After 50,000 outer iterations, we obtain an approximate solution with a relative $L^2$ error $4.034e\text{-}2$. As a comparison, the DNN-based least square is also applied to solve the same problem and the relative $L^2$ error is $1.015e\text{-}1$, which is much larger than the one by Friedrichs learning. In Figure \ref{highdim_errorlist}, we observe that DNN-based least square is not stable in optimization due to the curved discontinuity, and stops ultimately at a solution with a large error.

\begin{table}
\centering
\begin{tabular}{|c|c|c|c|c|c|c|c|}
\hline
  Parameters & $n$  & $m_s$ & $m_t$ &$N$  & $N_b$ & $\eta_s^p$ \\ \hline
  Value & $50,000$  & $150$ & $150$ &$45,000$ & $5,000$ & $3e\text{-}4$\\ \hline
  Parameters & $\eta_t^p$   & $\eta_s^{(0)}$ &$ \eta_t^{(0)}$ & $\nu_s$ & $\nu_t$ & parameter number \\\hline
 Value & $3e\text{-}3$  & $5e\text{-}5$ & $5e\text{-}4$ &$20,000$&$20,000$ & $115,050$\\ \hline
\end{tabular}
\vspace{0.25cm}
\caption{\em The parameters for the Friedrichs learning solver of the experiment in Section \ref{subsec:highdim}.}
\label{tab:highdim}
\end{table}

\begin{table}
\centering
\begin{tabular}{|c|c|c|c|c|c|c|}
  \hline
  Parameters &$n$ & $N$   & $\eta_s^{(0)}$ &$\nu_s$ &  $m_s$ \\ \hline
  Value &$50,000$ &$45,000$   & $1e\text{-}3$ & $20,000$ & $150$ \\ \hline
\end{tabular}
\vspace{0.25cm}
\caption{\em The parameters of the comparative experiment in Section \ref{subsec:highdim}.}
\label{tab:highdim-LS}
\end{table}

\subsection{Maxwell Equations}
\label{subsec:maxwell}
In the last example, we consider Maxwell equations \eqref{eqn:maxwell} defined in the domain $\Omega=[0,\pi]^3$.  Let $\bm H$ and $\bm E$  be the solutions of the Maxwell equations \eqref{eqn:maxwell}
with $\mu = \sigma = 1$. Let  $\bm f,\bm g  \in [L^2(\Omega)]^3$ be
$
\bm f =
(0,0,0)^{\intercal}$ and $
\bm g =(
3\sin y \sin z,
3\sin z \sin x,
3\sin x \sin y)^{\intercal}$.
The boundary condition is set as $\bm E\times \bm n = 0$, which is an ideal conductor boundary condition. The exact solutions to these equations are
$
\bm H^* =(
\sin x (\cos z - \cos y),
\sin y (\cos x - \cos z),
\sin x (\cos y - \cos x))^{\intercal}$  and $\bm E^* =
(\sin y \sin z,
\sin z \sin x,
\sin x \sin y)^{\intercal}$.
Considering test functions $(\varphi_{\bm H}^{\intercal}, \varphi_{\bm E}^{\intercal})^{\intercal}$ in the space $V^* = V$ mentioned in \eqref{coneboundary}, we set up DNNs to satisfy the boundary conditions $\varphi_{\bm E} \cdot \bm n = 0$ and $\varphi_{\bm H} \times \bm n = 0$, where $\bm n$ is the unit outward normal direction to the boundary. Note that the domain is a cube, the normal vector is parallel to one of the unit vectors.  The boundary condition above is indeed a Dirichlet  boundary. For example, $S_1 = \{x=\pi\} \cap \partial \Omega$ on the right surface, implies that $E_2 |_{S_1} = E_3 |_{S_1} = 0$. It is worth pointing out that the Dirichlet boundary for $(E_i^{\intercal}, (\varphi_{\bm H})_i)^{\intercal}$ closes the faces of the cube as shown in Figure \ref{sample_diagram2}. Here, we denote by $ E_i (i=1,2,3)$  the $i$-th component of the vector $\bm E$ and the same applies to other notations. 

To solve the Maxwell equations by Friedrichs learning, we initialize sub-networks of width $m_s$ for vector functions and  each sub-network decides one output value of the vector function. The test networks are set up similarly. We list all the parameters used in this experiment in Table \ref{tab:maxwell}. After $20,000$ outer iterations, we obtain an $L^2$ relative error $1.766e\text{-}2$ and an $L^\infty$ relative error $3.467e\text{-}2$.
Figure \ref{2D_maxwell_diff_E} and Figure \ref{2D_maxwell_diff_H} illustrate the absolute difference between $E_1$ and $(\phi_{\bm E})_1$ and the absolute difference between $H_1 $ and $(\phi_{\bm H})_1$ after $20,000$ outer iterations.


\begin{figure}[htbp]
    \subfigure[The boundary conditions of $(E_1, (\phi_{\bm H})_1)$.]{
	\begin{minipage}[t]{0.45\textwidth}
		\centering
		\includegraphics[width=0.75\linewidth]{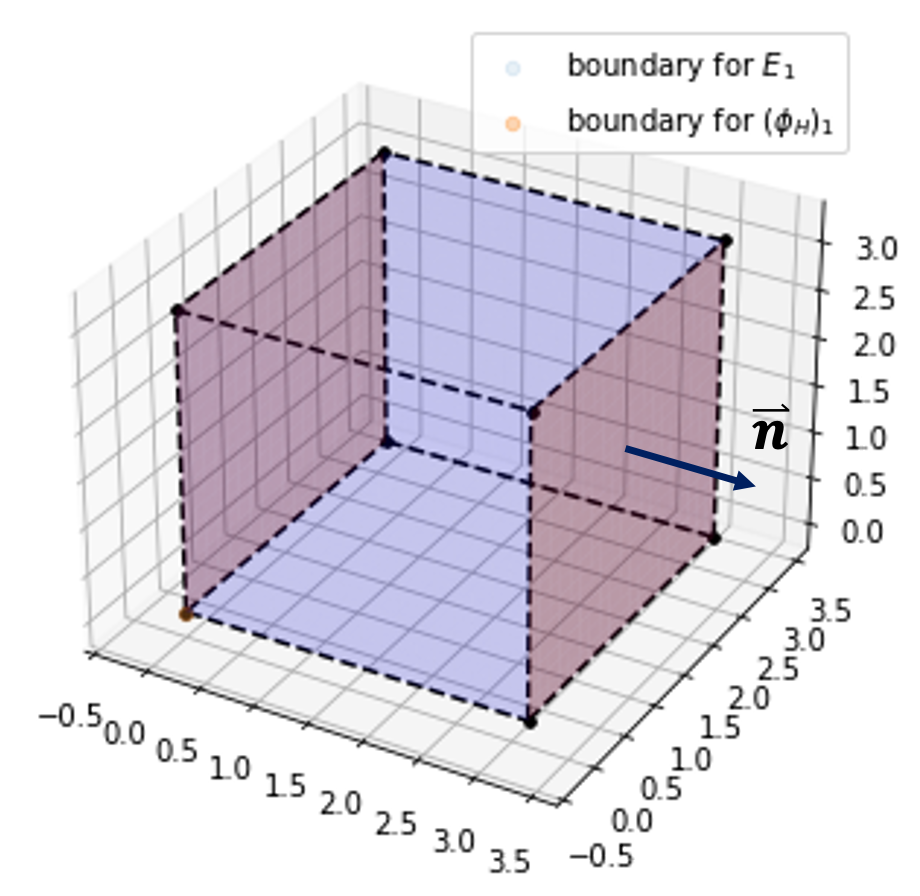}
        \label{sample_diagram2}
	\end{minipage}
	}
	\hfill
	\subfigure[The relative error versus iterations.]{
	\begin{minipage}[t]{0.45\textwidth}
		\centering
		\includegraphics[width=0.95\linewidth]{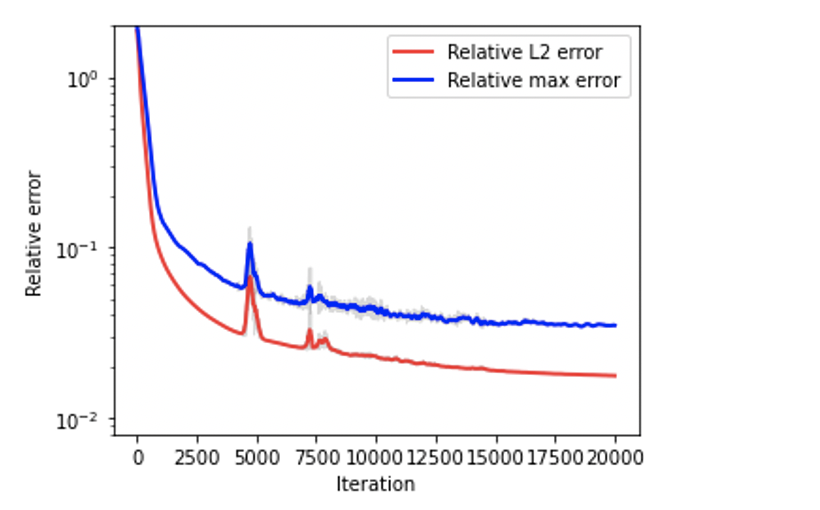}
        \label{fig_loss_maxwell}
	\end{minipage}
	}
   	\hfill
	\subfigure[The absolute difference between $E_1 $ and $(\phi_{\bm E})_1$ after $20,000$ outer iterations. ]{
	\begin{minipage}[t]{0.45\textwidth}
		\centering
		\includegraphics[width=0.8\linewidth]{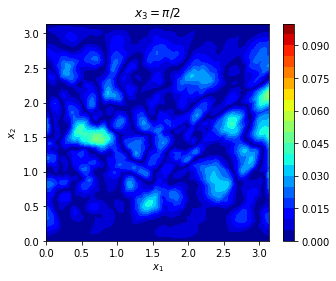}
        \label{2D_maxwell_diff_E}
	\end{minipage}
    }
    \hfill
	\subfigure[The absolute difference between $H_1 $ and $(\phi_{\bm H})_1$ after $20,000$ outer iterations. ]{
	\begin{minipage}[t]{0.45\textwidth}
		\centering
		\includegraphics[width=0.8\linewidth]{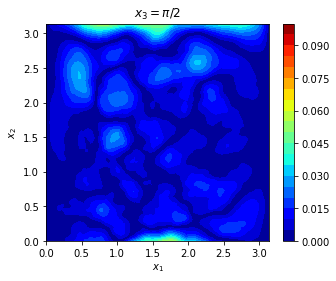}
        \label{2D_maxwell_diff_H}
	\end{minipage}
    }
	\hfill
\centering
\caption{Numerical results of Maxwell equations in \eqref{eqn:maxwell}.}
\end{figure}

\begin{table}
\centering
\begin{tabular}{|c|c|c|c|c|c|c|}
  \hline
  Parameters & $n$  & $m_s$ & $m_t$ &$N$  \\ \hline
  Value & $20,000$  & $250$  & $50$ &$50,000$ \\ \hline
  Parameters  & $\eta_s^{(0)}$ &$ \eta_t^{(0)}$ & $\nu_s$ & $\nu_t$ \\\hline
 Value & $3e\text{-}6$ & $3e\text{-}3$ &$8,000$&$15,000$\\ \hline
\end{tabular}
\vspace{0.25cm}
\caption{\em The parameters for Friedrichs learning solver of the experiment in Section \ref{subsec:maxwell}.}
\label{tab:maxwell}
\end{table}

\section{Conclusion}\label{sec:CC}
Friedrichs learning was proposed as a new deep learning methodology to learn the weak solutions of PDEs via Friedrichs seminal minimax formulation. Extensive numerical results imply that our mesh-free method provides reasonably accurate solutions for a wide range of PDEs defined on regular and irregular domains in various dimensions, where classical numerical methods may be difficult to be employed. In particular, Friedrichs learning infers the solution without the knowledge of the location of discontinuity when the solution is discontinuous. Our numerical experiments show that Friedrichs learning can solve PDEs with a discontinuous solution to $O(10^{-2})$ accuracy, while the DNN-based least square method can typically only get $O(10^{-1})$ accuracy. This demonstrates the advantage of the loss function in Friedrichs learning over the naive least square loss function. Compared with traditional FEM methods, Friedrichs learning performs as well as DGFEM with adaptive mesh when no prior knowledge about the discontinuous location is known. Friedrichs learning is better than LSFEM with adaptive mesh when no prior knowledge about the discontinuous location is known. In the future, it is interesting to develop adaptive Friedrichs learning to further reduce the error or the network size.
\vskip 0.2cm

{\bf Acknowledgements.} J. H. was partially supported by NSFC (Grant No. 12071289), the Strategic Priority Research Program of Chinese Academy of Sciences (Grant No. XDA25010402) and Shanghai Municipal Science and Technology Major Project (2021SHZDZX0102). C. W. was partially supported by National Science Foundation Award DMS-2136380 and DMS-2206333. H. Y. was partially supported by the NSF Award DMS-2244988 and DMS-2206333, ONR N00014-23-1-2007, and the NVIDIA GPU grant.

\bibliography{core}

\begin{thebibliography}{10}

\bibitem{Adams1975}
{\sc Robert~A. Adams}, {\em Sobolev spaces}, Pure and Applied Mathematics, Vol.
  65, Academic Press [Harcourt Brace Jovanovich, Publishers], New York-London,
  1975.

\bibitem{Al-Dujaili2019}
{\sc A.~Al-Dujaili, S.~Srikant, E.~Hemberg, and U.-M. O’Reilly}, {\em On the
  application of danskin’s theorem to derivative-free minimax problems}, in
  AIP Conference Proceedings, vol.~2070, AIP Publishing LLC, 2019, p.~020026.

\bibitem{Anthony2009book}
{\sc M.~Anthony and P.~L. Bartlett}, {\em Neural Network Learning: Theoretical
  Foundations}, Cambridge University Press, New York, NY, USA, 1st~ed., 2009.

\bibitem{Antonic2009}
{\sc N.~Antoni\'{c} and K.~Burazin}, {\em Graph spaces of first-order linear
  partial differential operators}, Math. Commun., 14 (2009), pp.~135--155.

\bibitem{Antonic2010}
\leavevmode\vrule height 2pt depth -1.6pt width 23pt, {\em Intrinsic boundary
  conditions for {F}riedrichs systems}, Comm. Partial Differential Equations,
  35 (2010), pp.~1690--1715.

\bibitem{Aubin2000}
{\sc J.-P. Aubin}, {\em Applied functional analysis}, Pure and Applied
  Mathematics (New York), Wiley-Interscience, New York, second~ed., 2000.
\newblock With exercises by Bernard Cornet and Jean-Michel Lasry, Translated
  from the French by Carole Labrousse.

\bibitem{Bao2020}
{\sc G.~Bao, X.~Ye, Y.~Zang, and H.~Zhou}, {\em Numerical solution of inverse
  problems by weak adversarial networks}, Inverse Problems, 36 (2020),
  pp.~115003, 31.

\bibitem{Barron1993}
{\sc A.~R. Barron}, {\em Universal approximation bounds for superpositions of a
  sigmoidal function}, IEEE Transactions on Information theory, 39 (1993),
  pp.~930--945.

\bibitem{Beck2021}
{\sc C.~Beck, S.~Becker, P.~Cheridito, A.~Jentzen, and A.~Neufeld}, {\em Deep
  splitting method for parabolic {PDE}s}, SIAM J. Sci. Comput., 43 (2021),
  pp.~A3135--A3154.

\bibitem{Beck2021_2}
{\sc C.~Beck, S.~Becker, P.~Grohs, N.~Jaafari, and A.~Jentzen}, {\em Solving
  the {K}olmogorov {PDE} by means of deep learning}, J. Sci. Comput., 88
  (2021), pp.~Paper No. 73, 28.

\bibitem{Berg2018}
{\sc J.~Berg and K.~Nystr{\"o}m}, {\em A unified deep artificial neural network
  approach to partial differential equations in complex geometries},
  Neurocomputing, 317 (2018), pp.~28--41.

\bibitem{Bui-Thanh2013}
{\sc T.~Bui-Thanh, L.~Demkowicz, and O.~Ghattas}, {\em A unified discontinuous
  {P}etrov-{G}alerkin method and its analysis for {F}riedrichs' systems}, SIAM
  J. Numer. Anal., 51 (2013), pp.~1933--1958.

\bibitem{Cai2020}
{\sc W.~Cai, X.~Li, and L.~Liu}, {\em A phase shift deep neural network for
  high frequency approximation and wave problems}, SIAM J. Sci. Comput., 42
  (2020), pp.~A3285--A3312.

\bibitem{Daskalakis2018}
{\sc C.~Daskalakis and I.~Panageas}, {\em The limit points of (optimistic)
  gradient descent in min-max optimization}, in Proceedings of the 32Nd
  International Conference on Neural Information Processing Systems, NIPS'18,
  USA, 2018, Curran Associates Inc., pp.~9256--9266.

\bibitem{Dissanayake1994}
{\sc M.~W. M.~G. Dissanayake and N.~Phan-Thien}, {\em Neural-network-based
  approximations for solving partial differential equations}, communications in
  Numerical Methods in Engineering, 10 (1994), pp.~195--201.

\bibitem{E2017}
{\sc W.~E, J.~Han, and A.~Jentzen}, {\em Deep learning-based numerical methods
  for high-dimensional parabolic partial differential equations and backward
  stochastic differential equations}, Commun. Math. Stat., 5 (2017),
  pp.~349--380.

\bibitem{E2019}
{\sc W.~E, C.~Ma, and L.~Wu}, {\em {Barron Spaces and the Compositional
  Function Spaces for Neural Network Models}}, arXiv e-prints, arXiv:1906.08039
  (2019).

\bibitem{E2018}
{\sc W.~E and B.~Yu}, {\em The deep {R}itz method: a deep learning-based
  numerical algorithm for solving variational problems}, Commun. Math. Stat., 6
  (2018), pp.~1--12.

\bibitem{Ehrhardt2008}
{\sc M.~Ehrhardt and R.~E. Mickens}, {\em A fast, stable and accurate numerical
  method for the {B}lack-{S}choles equation of {A}merican options}, Int. J.
  Theor. Appl. Finance, 11 (2008), pp.~471--501.

\bibitem{ern200601}
{\sc A.~Ern and J.-L. Guermond}, {\em Discontinuous {G}alerkin methods for
  {F}riedrichs' systems. {I}. {G}eneral theory}, SIAM J. Numer. Anal., 44
  (2006), pp.~753--778.

\bibitem{ern200602}
\leavevmode\vrule height 2pt depth -1.6pt width 23pt, {\em Discontinuous
  {G}alerkin methods for {F}riedrichs' systems. {II}. {S}econd-order elliptic
  {PDE}s}, SIAM J. Numer. Anal., 44 (2006), pp.~2363--2388.

\bibitem{ern2008}
\leavevmode\vrule height 2pt depth -1.6pt width 23pt, {\em Discontinuous
  {G}alerkin methods for {F}riedrichs' systems. {III}. {M}ultifield theories
  with partial coercivity}, SIAM J. Numer. Anal., 46 (2008), pp.~776--804.

\bibitem{Ern2007}
{\sc A.~Ern, J.-L. Guermond, and G.~Caplain}, {\em An intrinsic criterion for
  the bijectivity of {H}ilbert operators related to {F}riedrichs' systems},
  Comm. Partial Differential Equations, 32 (2007), pp.~317--341.

\bibitem{Friedrichs1958}
{\sc K.~O. Friedrichs}, {\em Symmetric positive linear differential equations},
  Comm. Pure Appl. Math., 11 (1958), pp.~333--418.

\bibitem{Gaikwad2009}
{\sc A.~Gaikwad and I.~M. Toke}, {\em Gpu based sparse grid technique for
  solving multidimensional options pricing pdes}, in Proceedings of the 2Nd
  Workshop on High Performance Computational Finance, WHPCF '09, New York, NY,
  USA, 2009, ACM, pp.~6:1--6:9.

\bibitem{Gobovic1994}
{\sc D.~Gobovic and M.~E. Zaghloul}, {\em Analog cellular neural network with
  application to partial differential equations with variable mesh-size}, in
  Proceedings of IEEE International Symposium on Circuits and Systems - ISCAS
  '94, vol.~6, May 1994, pp.~359--362 vol.6.

\bibitem{Grisvard1985}
{\sc P.~Grisvard}, {\em Elliptic Problems in Nonsmooth Domains}, Pitman,
  Boston, 1985.

\bibitem{gu2021structure}
{\sc Y.~Gu, C.~Wang, and H.~Yang}, {\em Structure probing neural network
  deflation}, J. Comput. Phys., 434 (2021), pp.~Paper No. 110231, 21.

\bibitem{gu2021selectnet}
{\sc Y.~Gu, H.~Yang, and C.~Zhou}, {\em Select{N}et: self-paced learning for
  high-dimensional partial differential equations}, J. Comput. Phys., 441
  (2021), pp.~Paper No. 110444, 18.

\bibitem{2020Ingo}
{\sc I.~G\"{u}hring, G.~Kutyniok, and P.~Petersen}, {\em Error bounds for
  approximations with deep {R}e{LU} neural networks in {$W^{s,p}$} norms},
  Anal. Appl. (Singap.), 18 (2020), pp.~803--859.

\bibitem{2021Ingo}
{\sc I.~Gühring and M.~Raslan}, {\em Approximation rates for neural networks
  with encodable weights in smoothness spaces}, Neural Networks, 134 (2021),
  p.~107–130.

\bibitem{Han2018}
{\sc J.~Han, A.~Jentzen, and W.~E}, {\em Solving high-dimensional partial
  differential equations using deep learning}, Proc. Natl. Acad. Sci. USA, 115
  (2018), pp.~8505--8510.

\bibitem{Hanada2019}
{\sc K.~Hanada, T.~Wada, and Y.~Fujisaki}, {\em A restart strategy with time
  delay in distributed minimax optimization}, in Theory and Practice of
  Computation: Proceedings of Workshop on Computation: Theory and Practice
  WCTP2017, World Scientific, 2019, pp.~89--100.

\bibitem{Hassan2018}
{\sc R.~Hassan, L.~Mingrui, L.~Qihang Lin, and Y.~Tianbao}, {\em Non-convex
  min-max optimization: Provable algorithms and applications in machine
  learning}, ArXiv, abs/1810.02060 (2018).

\bibitem{HeZhangRenSun2016}
{\sc K.~He, X.~Zhang, S.~Ren, and J.~Sun}, {\em Deep residual learning for
  image recognition}, in Proceedings of the IEEE conference on computer vision
  and pattern recognition, 2016, pp.~770--778.

\bibitem{Hinton2012}
{\sc G.~Hinton, N.~Srivastava, and K.~Swersky}, {\em Neural networks for
  machine learning lecture 6a overview of mini-batch gradient descent}, Cited
  on, 14 (2012), p.~2.

\bibitem{hon2022simultaneous}
{\sc S.~Hon and H.~Yang}, {\em Simultaneous neural network approximation for
  smooth functions}, Neural Networks, 154 (2022), pp.~152--164.

\bibitem{Houston2000}
{\sc P.~Houston, C.~Schwab, and E.~S\"{u}li}, {\em Stabilized {$hp$}-finite
  element methods for first-order hyperbolic problems}, SIAM J. Numer. Anal.,
  37 (2000), pp.~1618--1643.

\bibitem{Hu2009RandomRI}
{\sc X.~Hu, R.~Shonkwiler, and M.~C. Spruill}, {\em Random restarts in global
  optimization},  (2009).

\bibitem{Huang2020}
{\sc J.~Huang, H.~Wang, and H.~Yang}, {\em Int-{D}eep: a deep learning
  initialized iterative method for nonlinear problems}, J. Comput. Phys., 419
  (2020), pp.~109675, 24.

\bibitem{Hutzenthaler2020}
{\sc M.~Hutzenthaler, A.~Jentzen, T.~Kruse, and T.~A. Nguyen}, {\em A proof
  that rectified deep neural networks overcome the curse of dimensionality in
  the numerical approximation of semilinear heat equations}, Partial Differ.
  Equ. Appl., 1 (2020), pp.~Paper No. 10, 34.

\bibitem{Hutzenthaler2020_1}
{\sc M.~Hutzenthaler, A.~Jentzen, T.~Kruse, T.~A. Nguyen, and P.~von
  Wurstemberger}, {\em Overcoming the curse of dimensionality in the numerical
  approximation of semilinear parabolic partial differential equations}, Proc.
  A., 476 (2020), pp.~20190630, 25.

\bibitem{Hutzenthaler2020_2}
{\sc M.~Hutzenthaler, A.~Jentzen, and P.~von Wurstemberger}, {\em Overcoming
  the curse of dimensionality in the approximative pricing of financial
  derivatives with default risks}, Electron. J. Probab., 25 (2020), pp.~Paper
  No. 101, 73.

\bibitem{Jagtap2020}
{\sc A.~D. Jagtap and G.~E. Karniadakis}, {\em Extended physics-informed neural
  networks ({XPINN}s): a generalized space-time domain decomposition based deep
  learning framework for nonlinear partial differential equations}, Commun.
  Comput. Phys., 28 (2020), pp.~2002--2041.

\bibitem{Jensen2004}
{\sc M.~Jensen}, {\em {Discontinuous Galerkin Methods for Friedrichs’ Systems
  with Irregular Solutions}}, Ph.D. thesis, University of Oxford, Oxford,
  (2004).

\bibitem{Joubert1994}
{\sc W.~Joubert}, {\em On the convergence behavior of the restarted {GMRES}
  algorithm for solving nonsymmetric linear systems}, Numer. Linear Algebra
  Appl., 1 (1994), pp.~427--447.

\bibitem{Kharazmi2021}
{\sc E.~Kharazmi, Z.~Zhang, and G.~E.~M. Karniadakis}, {\em {$hp$}-{VPINN}s:
  variational physics-informed neural networks with domain decomposition},
  Comput. Methods Appl. Mech. Engrg., 374 (2021), pp.~Paper No. 113547, 25.

\bibitem{Khoo2017SolvingPP}
{\sc Y.~Khoo, J.~Lu, and L.~Ying}, {\em Solving parametric {PDE} problems with
  artificial neural networks}, European J. Appl. Math., 32 (2021),
  pp.~421--435.

\bibitem{Kingma2014}
{\sc D.~P. Kingma and J.~Ba}, {\em {Adam: A Method for Stochastic
  Optimization}}, arXiv e-prints,  (2014).

\bibitem{Lagaris1998}
{\sc I.~E. Lagaris, A.~Likas, and D.~I. Fotiadis}, {\em Artificial neural
  networks for solving ordinary and partial differential equations}, IEEE
  Transactions on Neural Networks, 9 (1998), pp.~987--1000.

\bibitem{Lee1990}
{\sc H.~Lee, Hyuk, and I.~S. Kang}, {\em Neural algorithm for solving
  differential equations}, J. Comput. Phys., 91 (1990), pp.~110--131.

\bibitem{Lee2002}
{\sc T.T. Lee, F.Y. Wang, and R.B. Newell}, {\em Robust model-order reduction
  of complex biological processes}, Journal of Process Control, 12 (2002),
  pp.~807 -- 821.

\bibitem{Li2019}
{\sc K.~Li, K.~Tang, T.~Wu, and Q.~Liao}, {\em D3m: A deep domain decomposition
  method for partial differential equations}, IEEE Access, 8 (2019),
  pp.~5283--5294.

\bibitem{Li2019_2}
{\sc Q.~Li, B.~Lin, and W.~Ren}, {\em Computing committor functions for the
  study of rare events using deep learning}, The Journal of Chemical Physics,
  151 (2019), p.~054112.

\bibitem{Liu2020}
{\sc Q.~Liu and S.~Zhang}, {\em Adaptive least-squares finite element methods
  for linear transport equations based on an {$\rm H(div)$} flux
  reformulation}, Comput. Methods Appl. Mech. Engrg., 366 (2020), pp.~113041,
  25.

\bibitem{liu2020multiscale}
{\sc Z.~Liu, W.~Cai, and Z.-Q.~J. Xu}, {\em Multi-scale deep neural network
  ({M}scale{DNN}) for solving {P}oisson-{B}oltzmann equation in complex
  domains}, Commun. Comput. Phys., 28 (2020), pp.~1970--2001.

\bibitem{Montanelli2019_2}
{\sc H.~Montanelli and H.~Yang}, {\em Error bounds for deep relu networks using
  the kolmogorov--arnold superposition theorem}, Neural Networks, 129 (2020),
  pp.~1--6.

\bibitem{Montanelli2019}
{\sc H.~Montanelli, H.~Yang, and Q.~Du}, {\em Deep {R}e{LU} networks overcome
  the curse of dimensionality for generalized bandlimited functions}, J.
  Comput. Math., 39 (2021), pp.~801--815.

\bibitem{Nouiehed2019}
{\sc M.~Nouiehed, M.~Sanjabi, T.~Huang, J.~D. Lee, and M.~Razaviyayn}, {\em
  Solving a class of non-convex min-max games using iterative first order
  methods}, Advances in Neural Information Processing Systems, 32 (2019).

\bibitem{RAISSI2019686}
{\sc M.~Raissi, P.~Perdikaris, and G.~E. Karniadakis}, {\em Physics-informed
  neural networks: a deep learning framework for solving forward and inverse
  problems involving nonlinear partial differential equations}, J. Comput.
  Phys., 378 (2019), pp.~686--707.

\bibitem{shen2021deep}
{\sc Z.~Shen, H.~Yang, and S.~Zhang}, {\em Deep network approximation:
  Achieving arbitrary accuracy with fixed number of neurons}, arXiv preprint
  arXiv:2107.02397,  (2021).

\bibitem{Shen2021}
\leavevmode\vrule height 2pt depth -1.6pt width 23pt, {\em Deep network with
  approximation error being reciprocal of width to power of square root of
  depth}, Neural Comput., 33 (2021), pp.~1005--1036.

\bibitem{Shen2021_2}
\leavevmode\vrule height 2pt depth -1.6pt width 23pt, {\em Neural network
  approximation: Three hidden layers are enough}, Neural Networks, 141 (2021),
  pp.~160--173.

\bibitem{SIEGEL2020313}
{\sc J.~W. Siegel and J.~Xu}, {\em Approximation rates for neural networks with
  general activation functions}, Neural Networks, 128 (2020), pp.~313--321.

\bibitem{Sirignano2018}
{\sc J.~Sirignano and K.~Spiliopoulos}, {\em D{GM}: a deep learning algorithm
  for solving partial differential equations}, J. Comput. Phys., 375 (2018),
  pp.~1339--1364.

\bibitem{Srinivasa2017}
{\sc C.~Srinivasa, I.~Givoni, S.~Ravanbakhsh, and B.~J. Frey}, {\em Min-max
  propagation}, in Advances in Neural Information Processing Systems 30,
  I.~Guyon, U.~V. Luxburg, S.~Bengio, H.~Wallach, R.~Fergus, S.~Vishwanathan,
  and R.~Garnett, eds., Curran Associates, Inc., 2017, pp.~5565--5573.

\bibitem{Wales2003}
{\sc D.~J. Wales and J.~P.~K Doye}, {\em Stationary points and dynamics in
  high-dimensional systems}, The Journal of chemical physics, 119 (2003),
  pp.~12409--12416.

\bibitem{Yserentant2005}
{\sc H.~Yserentant}, {\em Sparse grid spaces for the numerical solution of the
  electronic {S}chr\"{o}dinger equation}, Numer. Math., 101 (2005),
  pp.~381--389.

\bibitem{Zang2020}
{\sc Y.~Zang, G.~Bao, X.~Ye, and H.~Zhou}, {\em Weak adversarial networks for
  high-dimensional partial differential equations}, J. Comput. Phys., 411
  (2020), pp.~109409, 14.

\end{thebibliography}
\bibliographystyle{siam}

\end{document}